\documentclass{aims}
\usepackage{amsmath}
  \usepackage{paralist}
  \usepackage{graphics} 
  \usepackage{epsfig} 
\usepackage{graphicx}  \usepackage{epstopdf}
 \usepackage[colorlinks=true]{hyperref}
 \usepackage{moreverb,url}
\usepackage[tableposition=top]{caption}
\usepackage{multirow}
\usepackage{mathtools}
\usepackage{lmodern}
\usepackage{arydshln}
\def\block(#1,#2)#3{\multicolumn{#2}{c}{\multirow{#1}{*}{$ #3 $}}}
\usepackage[font=footnotesize]{caption}
    \usepackage{booktabs}
    \usepackage[dvipsnames]{xcolor}
    \usepackage{tikz}
\usepackage{ifpdf}
\usepackage{multicol}
\hypersetup{urlcolor=blue, citecolor=red}

  \def\ppages{1--30}
    \def\DOI{}

\newtheorem{theorem}{Theorem}[section]

\newtheorem{lemma}[theorem]{Lemma}

\theoremstyle{definition}

\newtheorem{remark}{Remark}

\def\bp3{{\beta_{33}}}
\newcommand{\mb}{\mathbf}
\newcommand{\mc}{\mathcal}

\def\mX{{\mathbb X}}
\def\Ltwo{{\mathbb L}^2 }
\title[Modeling and stabilization results for a piezoelectric  composite] 
      {Dynamic and electrostatic modeling for a piezoelectric smart composite and related stabilization results}

\author[Ahmet \"Ozkan \"Ozer]{}


 \email{ozkan.ozer@wku.edu}

\thanks{The  author is supported by the Western Kentucky University startup research grant.}

\usepackage{fancyhdr}
\usepackage{lipsum}

\begin{document}
\maketitle

\centerline{\scshape Ahmet \"Ozkan \"Ozer$^*$}
\medskip
{\footnotesize
 \centerline{Department of Mathematics, Western Kentucky University}
   \centerline{ Bowling Green, KY 42101, USA}
} 

\medskip


\bigskip



\begin{abstract}
A cantilevered piezoelectric smart composite beam, consisting of perfectly bonded elastic, viscoelastic and piezoelectric layers,  is considered. The piezoelectric layer is actuated by a voltage source. Both fully dynamic and electrostatic approaches, based on Maxwell's equations, are used to model the piezoelectric layer.  We obtain (i) fully-dynamic and electrostatic Rao-Nakra type models by assuming that the viscoelastic layer has a negligible weight and stiffness, (ii) fully-dynamic and electrostatic Mead-Marcus type models  by neglecting the in-plane and rotational inertia terms. Each model is a  perturbation of the corresponding classical smart composite beam model.  These models are written  in the state-space form,   the existence and uniqueness of solutions are obtained in  appropriate Hilbert spaces.  Next, the stabilization problem for each closed-loop system, with a thorough analysis, is investigated for the natural $B^*-$type state feedback controllers. The fully dynamic  Rao-Nakra  model with four state feedback controllers is shown to be not asymptotically stable for certain choices of material parameters  whereas the electrostatic model is exponentially stable with only three state  feedback controllers (by the spectral multipliers method). Similarly, the fully dynamic Mead-Marcus model lacks of asymptotic stability for certain solutions whereas the electrostatic model  is exponentially stable by  only one state feedback controller.
\end{abstract}


\section{Introduction}
A piezoelectric smart composite beam consisting of a stiff elastic layer  \textcircled{1}, a complaint (viscoelastic) layer  \textcircled{2}, and a piezoelectric layer  \textcircled{3} is considered in this paper, see Fig. \ref{ACL}.
The piezoelectric layer  \textcircled{3} is also an elastic beam  with  electrodes at its top and bottom surfaces, insulated at the edges (to prevent fringing effects), and connected to an external electric circuit. (See Figure \ref{ACL}).  As the electrodes are subjected to a voltage source, the piezoelectric layer  compresses or extends, inducing a bending moment in the composite structure.
The electrostatic assumption (due to Maxwell's equations) is widely used to model  the single piezoelectric layer  which entirely ignores the dynamic effects for the Maxwells's magnetic equations, see i.e. \cite{Cao-Chen,Smith,Tiersten}. In fact, even though it is minor in comparison to the mechanical, the dynamic electromagnetic effects have a dramatic effect on the control of these materials \cite{Ozkan1,V-S3,V-S4}. Many control approaches are also available to control piezoelectric beams such as feedback, feedforward, sensorless, etc.  \cite{Choi,M-F}.

\begin{figure}[htp]
\centering
\includegraphics[width=4.5in]{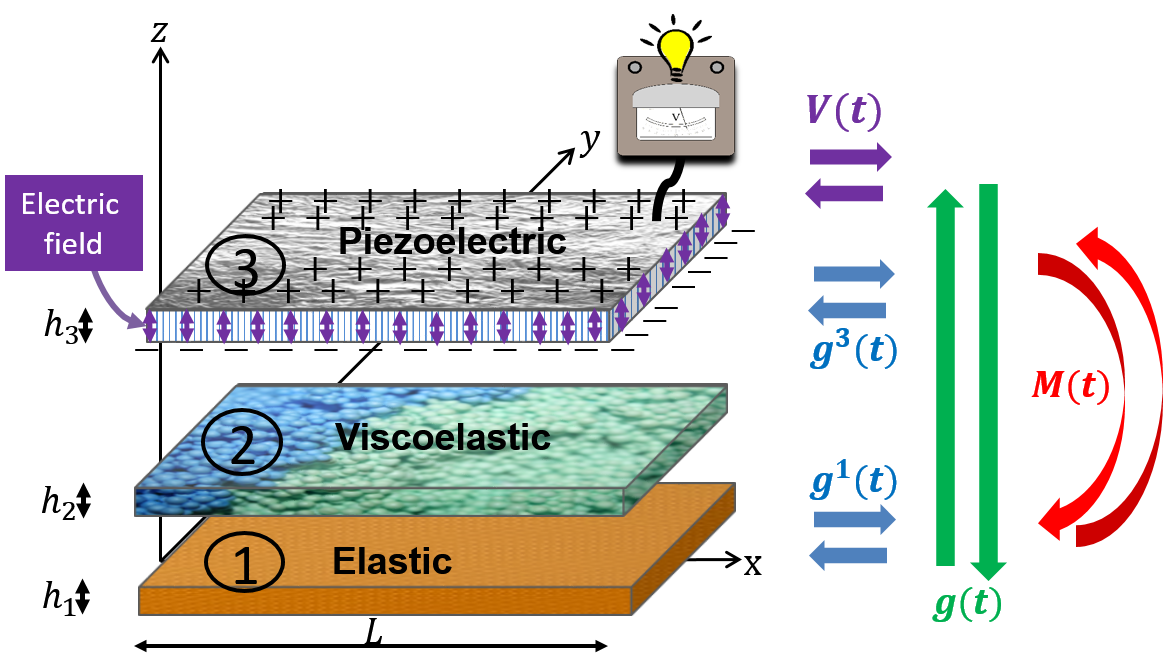}
\caption{A voltage-actuated piezoelectric smart composite of length $L$  with thicknesses $h_1, h_2, h_3$ for its layers { \textcircled{1},  \textcircled{2}, \textcircled{3},} respectively. The longitudinal motions of top and bottom layers  \textcircled{1} and  \textcircled{3} are controlled by $g^1(t),g^3(t), V(t),$ and the bending motions (of the whole composite) are controlled by $M(t), g(t).$ For the fully-dynamic models, written in the state-space formulation $\dot\varphi=\mc A \varphi + B u(t)$, the $B^*-$type observation for the piezoelectric layer naturally corresponds to the total induced current at its electrodes. It  is more physical in terms of practical applications. As well, measuring the total induced current at the electrodes of the piezoelectric layer is easier than measuring displacements or the velocity of the composite at one end of the beam, i.e. see \cite{Baz,Chee,Miller}.} \label{ACL}
 \end{figure}

The models of piezoelectric smart composites in the literature all assume the electrostatic assumption for its piezoelectric layer. These models differ by the assumptions for the viscoelastic layer and the geometry of the composite, i.e. see \cite{Trindade}.  The  models are either a Mead-Marcus (M-M) type or a Rao-Nakra (R-N) type  as  obtained  in \cite{Baz,Lam}, respectively. The M-M model only describes the bending motion, and the R-N model describes bending and longitudinal motions all together. These models reduce to the classical counterparts once the piezoelectric strain is taken to be zero \cite{Mead,Rao}. The active boundary feedback stabilization of the classical R-N model (having no piezoelectric layer) is only investigated  for hinged \cite{O-Hansen4} and clamped-free \cite{Wang} boundary conditions. The exact controllability of the M-M and R-N models are shown for the fully clamped, fully hinged, and clamped-hinged models \cite{O-Hansen1,O-Hansen3}. The exponential stability in the existence of the passive distributed ``shear" damping term is investigated for the  R-N and M-M models (\cite{A-H,W-G}) .

Asymptotic stabilization of piezoelectric smart composite models are investigated   in \cite{Baz,Lam} for various PID-type feedback controllers and a shear-type distributed damping.   The exponential stability of the electrostatic M-M and R-N models for clamped-free boundary conditions has been open problems for more than a decade.   Recently,   exponential stability of the electrostatic R-N model is shown  by using four type feedback controllers \cite{Ozkan3}, two for the longitudinal motions and two for the bending motions,  by using the compact perturbation argument \cite{Trigg} together with the use of spectral multipliers \cite{O-Hansen4}. Exponential stability with only three controllers is shown  by using a tedious spectral-theoretic approach  \cite{Y-W}. The fully dynamic R-N model  is also shown to be not asymptotically stable for many choices of material parameters by using the natural $B^*-$type feedback controllers  \cite{Ozkan3}. The charge-actuated electrostatic counterparts are shown to be exponentially stable whereas the current-actuated electrostatic counterparts can only be asymptotically stabilizable \cite{Ozer18b}.

In this paper, we develop a novel modeling strategy to model  a cantilevered piezoelectric smart composite with/without the magnetic effects due to the Maxwell's equations. We mainly use the methodology developed in \cite{Hansen3,Ozkan3}. We obtain two  models:
\begin{itemize}
  \item [(i)] A  fully dynamic Rao-Nakra-type model by assuming that the stiffness and the weight of the middle layer are negligible.
  \item [(ii)] A fully dynamic Mead-Marcus-type model by assuming that in-plane and rotational inertia terms are negligible.
\end{itemize}
By discarding the dynamic electromagnetic magnetic effects  (electro-magnetic kinetic energy) in these models,  electrostatic Rao-Nakra and Mead-Marcus models are simply derived.
The voltage control is part of the charge boundary conditions corresponding to the dynamic charge equation. The fully dynamic Mead-Marcus type model is novel in the literature. All models are shown to be well-posed and are written in the state-space formulation.

For each model, natural  $B^*-$type state feedback controllers are considered. We summarize our findings as the following:
\begin{itemize}
  \item [I.]The  fully dynamic closed-loop Rao-Nakra-type model is shown to be asymptotically stable for the inertial sliding solutions, see Theorem \ref{lack1}. It is simply because the outer layers of the composite are made of different materials, piezoelectric and elastic, and therefore,  the top and the bottom layers have different speeds of wave propagation. In fact, if the outer layers are identical piezoelectric beams as in \cite{Ozkan3}, asymptotic stability fails for inertial sliding solutions.
       \item [II.] In contrast to the result above, the fully dynamic Mead-Marcus-type model  is not asymptotically stable for inertial-sliding solutions if the material parameters satisfy a number-theoretical condition, see Theorem \ref{lack}.
  \item [III.] The electrostatic Rao-Nakra-type model  is recently shown to be exponentially stable with four controllers in \cite{Ozkan3}. This result is improved in Theorem \ref{RN-elec} by  eliminating one redundant controller. The proof uses higher order spectral multipliers  as in \cite{B-Rao}, see Lemma \ref{xyz}. In fact, the same result is recently announced in \cite{Y-W} by a different approach, which requires to determine the spectrum of the closed-loop system.

   \item  [IV.] The electrostatic Mead-Marcus-type model  is exponentially stable by only one feedback controller. Similarly, the proof mainly uses the spectral multipliers, see Theorem \ref{MM-mult}. Our result rigorously proves the findings of \cite{Baz} where only the asymptotic stability is claimed without a proof.
\end{itemize}

Note that modeling and well-posedness results for only the R-N model  are briefly announced in \cite{Ozkan5}.

\section{Modeling assumptions: Classical sandwich beam theory and electromagnetism}
The piezoelectric smart composite beam considered in this paper is consisting of a stiff layer, a complaint layer, and a piezoelectric layer, see Figure \ref{ACL}. The composite occupies the
region $\Omega=\Omega_{xy}\times (0, h):=[0,L]\times [-b,b] \times (0,h)$ at equilibrium.
The total thickness $h$ is assumed to be small in comparison to the dimensions of $\Omega_{xy}$.
 The layers are indexed from 1 to 3 from the stiff  layer to the piezoelectric layer, respectively.

Let $0=z_0<z_1<z_2<z_3=h, $ with
 $$h_i=z_i-z_{i-1}, \quad i=1,2,3.$$
We use the rectangular coordinates $X=(x,y)$ to denote points in $\Omega_{xy},$ and  $(X, z)$ to denote points in $\Omega = \Omega^{\rm s} \cup \Omega^{\rm ve} \cup \Omega^{\rm p} $, where $\Omega^{\rm s}, \Omega^{\rm ve},$ and $\Omega^{\rm p}$ are the reference configurations of the stiff, viscoelastic, and piezoelectric layer, respectively, and they are given by
\begin{eqnarray}
\nonumber &&\Omega^{\rm s}=\Omega_{xy}\times (z_0,z_1),~~ \Omega^{\rm ve}=\Omega_{xy}\times (z_1,z_2), ~~\Omega^{\rm p}=\Omega_{xy}\times (z_2,z_3).
\end{eqnarray}

For $(X,z)\in \Omega,$ let $U(X,z) = ({U_1,U_2,U_3})(X, z)$ denote the displacement vector of the point
(from reference configuration). For the  beam theory,  all displacements  are assumed to be independent of $y-$coordinate, and $U_2\equiv 0.$ The transverse displacements is  $w(x,y,z)= U_3(x)= w^i(x)$ for  any $i$ and $x\in [0,L].$ Define $u^i(x,y,z)=U_1(x,0,z_i)=u^i(x)$  for $i=0,1,2,3$ and  for all $x\in (0,L).$

 Define
$$\vec \psi=[\psi^1,\psi^2, \psi^3]^{\rm T}, \quad \vec \phi=[\phi^1,\phi^2,\phi^3]^{\rm T}, \quad \vec v=[v^1, v^2, v^3]^{\rm T}$$ where
\begin{eqnarray}\label{defs1} &&  \psi^i=\frac{u^i-u^{i-1}}{h_i}, \quad \phi^i= \psi^i + w_x,  ~~ v^i= \frac{u^{i-1}+u^i}{2}, \quad i = 1, 2, 3.
\end{eqnarray}
where $\psi^{i}$   is the total rotation angles
(with negative orientation) of the deformed filament within the $i^{\rm th}$ layer in the
$x-z$ plane, $\phi^i$ is the (small angle approximation for
the) shear angles within each layer, $v^i$ is the longitudinal displacement of the center line of the $i^{\rm th}$ layer.

 For the middle layer, we apply
Mindlin-Timoshenko small displacement assumptions, while for the outer layers Kirchhoff small displacement assumptions are applied. Therefore,
\begin{eqnarray}\label{defs3}\phi^1=\phi^3=0, \quad  \psi^1=\psi^3=-w_x, \quad \phi^2=\psi^2+ w_x.\end{eqnarray}

\begin{figure}[htp]
\centering
\includegraphics[width=4.5in]{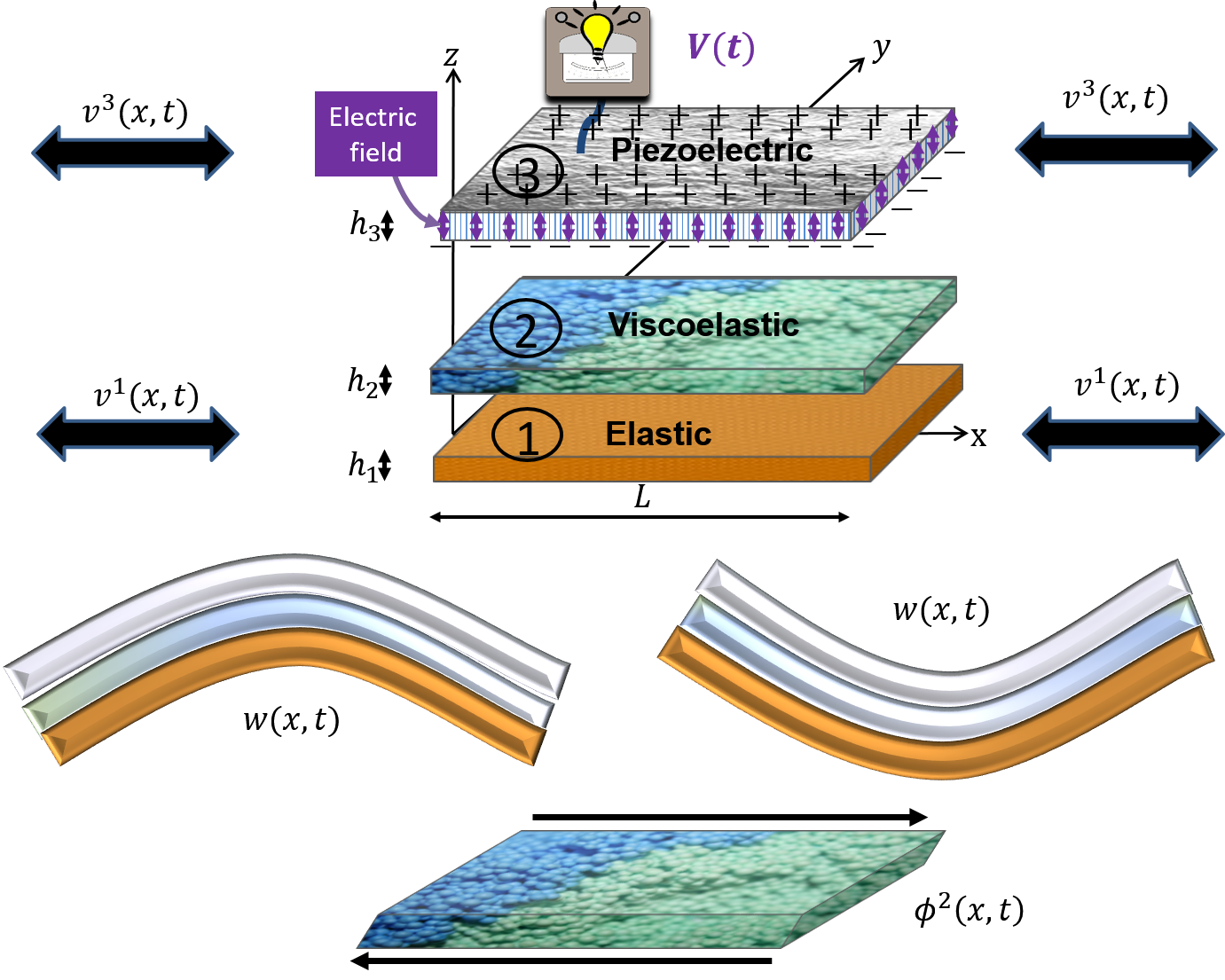}
\caption{ The equations of motion describing the overall ``small" vibrations on the composite beam are dictated by the variables $v^1(x,t),v^3(x,t),w(x,t),\phi^2(x,t)$ which correspond to the longitudinal vibrations of Layers  \textcircled{1} and \textcircled{3}, bending of the composite  \textcircled{1}-\textcircled{2}-\textcircled{3}, and shear of Layer \textcircled{2}.} \label{ACL2}
 \end{figure}
Let $G_2$ be the shear modulus of the viscoelastic layer. Defining ${\hat z}^i = \frac{z^{i-1}+z^i}{2},$ and
\begin{eqnarray}
\nonumber  && \alpha^1=c_{11}^1, ~~ \alpha^2=c_{11}^2,~~\alpha^3=\alpha_1^3 + \gamma^2 \beta, ~\alpha_1^3=c_{11}^3,\\
\label{coef}&&  \gamma=\gamma_{31},~\gamma_1=\gamma_{15},~\beta=\frac{1}{\varepsilon_{33}}, ~\beta_{1}=\frac{1}{\varepsilon_{11}},
\end{eqnarray}
 where  $c_{ij}^k$ are  elastic stiffness coefficients of each layer, and $\gamma_{ij}$ and $\varepsilon_{ij},$ are  piezoelectric and   permittivity coefficients for the piezoelectric layer. Refer to \cite{Ozer18c} for sample piezoelectric constants.
 The displacement field, strains,  and the constitutive relationships for each layer are given in Table \ref{table:displ-II}.
  \bgroup
\def\arraystretch{1.4}
\begin{table}[htp]
\centering
\begin{tabular}{|p{3.6cm}|p{8cm}p{0cm}|} \hline
Layers&   Displacements, Stresses, Strains, &  \\
&   Electric fields,  and Electric displacements&  \\ \hline
\multirow{3}{*}{Layer {\Large \textcircled{1}} - Elastic }&$U_1^1(x,z)=v^1(x)- (z-\hat z_1)w_x,~~U_3(x,z)=w(x)$ &  \\ \cline{2-3}
&$S_{11}=\frac{\partial v^1}{\partial x}+\frac{1}{2}(w_x)^2- (z-\hat z_1) \frac{\partial^2 w}{\partial x^2},~~S_{13}=0$ &  \\ \cline{2-3}
&$ T_{11}=\alpha_1S_{11},~~T_{13}=T_{12}=T_{23}=0$ &  \\ \cline{2-3}
\hline
\multirow{3}{*}{Layer {\Large \textcircled{2}} - Viscoelastic } &$U_1^2(x,z)=v^2(x)+ (z-\hat z_2)\psi^2(x),~~U_3(x,z)=w(x)$ &  \\ \cline{2-3}
&$ S_{11}=\frac{\partial v^2}{\partial x}- (z-\hat z_i) \frac{\partial \psi^2}{\partial x},~~S_{13}=\frac{1}{2}\phi^2$ &  \\ \cline{2-3}
&$ T_{11}=\alpha_1^2S_{11},~~T_{13}= 2G_{2} S_{13},$ ~~$T_{12}=T_{23}=0$ &  \\ \cline{2-3}
\hline
\multirow{4}{*}{Layer {\Large \textcircled{3}} - Piezoelectric }&$U_1^3(x,z)=v^3(x)- (z-\hat z_3)w_x,~~U_3(x,z)=w(x)$ &  \\ \cline{2-3}
&$ S_{11}=\frac{\partial v^3}{\partial x}+\frac{1}{2}(w_x)^2- (z-\hat z_3) \frac{\partial^2 w}{\partial x^2},~~S_{13}=0$ &  \\ \cline{2-3}
&$ T_{11}=\alpha^3 S_{11}-\gamma\beta D_3,~~T_{13}=T_{12}=T_{23}=0$ & \\ \cline{2-3}
& $E_1=\beta_{1}D_1,~~E_3=-\gamma\beta S_{11}+\beta D_3$  &  \\ \hline
\end{tabular}
\caption{ Linear  constitutive relationships for each layer. $U_i^j, T_{ij},$ $ S_{ij},$ $D_i,$ and $E_i$   denote displacements, the stress tensor, strain tensor, electrical displacement, and electric field for $i, j = 1, 2, 3.$}
\label{table:displ-II}
\end{table}
We follow the dynamic approach in \cite{O-M1,Ozkan3} to include the electromagnetic effects for the piezoelectric layer \textcircled{3} . The magnetic field $\textrm{B}$ is perpendicular to the $x-z$ plane, and therefore,  $\textrm{B}_2(x)$ is the only non-zero component. Assuming  $E_1=D_1=0,$   and by $\frac{ d \textrm{B}_2}{dx}=-\mu \dot D_3,$ we obtain the total induced current accumulated $[0,x]\in [0,L]$ portion of  the piezoelectric beam $\textrm{B}_2=-\mu\int_0^x \dot D_3(\xi, z, t) ~d\xi.$ This is physical since $B_2(0)=0$  at the clamped end of the beam and $B_2(L)=-\mu\int_0^L \dot D_3(\xi, z, t) ~d\xi$ at the free end of the beam. Analogously, we define $p=\int_0^x  D_3(\xi,  t) ~d\xi$ to be the total charge of the piezoelectric beam at $x\in (0,L]$ so that $p_x=D_3,$ and $p(0)=0.$

Assume that the beam is subject to a distribution of boundary forces $(\tilde g^1, \tilde g^3, \tilde g)$  along its edge $x=L,$ see Fig. \ref{ACL}.
Now define
\begin{eqnarray}
 \nonumber   \left.
\begin{array}{ll}
 g^i(x,t )=\int_{0}^{z_1} \tilde g^i(x,z,t)~dz,  ~g(x,t )=\int_{0}^{h} \tilde g(x,z,t)~dz,& \\
 ~~m_i=\int_{z_{i-1}}^{z_i} (z-\hat z_i) \tilde g^i(x,z,t)~dz, \quad  i=1,3
 \end{array} \right.
\end{eqnarray}
to be the external force resultants defined as in \cite{Lagnese-Lions}. For our model, it is appropriate to assume that  $\tilde g^1, \tilde g^3$ are independent of $z,$ Let $V(t)$ be the voltage applied at the electrodes of the piezoelectric layer. By using Table \ref{table:displ-II} (with the assumption $D_1=0$), the  Lagrangian for the ACL beam is
 \begin{eqnarray} \nonumber \mb{L}= \int_0^T \left[\mb{K}-(\mb{P}+\mb{E})+\mb B +\mb{W}\right]~dt\end{eqnarray}
 where
  \begin{eqnarray}
\label{energies}  \begin{array}{ll}
 \mb{K}= \frac{1}{2} \int_0^L \left[\rho_1h_1(\dot v^1)^2+\rho_2h_2(\dot v^2)^2+\rho_3h_3(\dot v^3)^2 + \rho_2h_2 (\dot \psi^2)^2\right. &\\
   \quad \left.+ \left(\rho_1h_1 + \rho_3h_3\right) \dot w_x^2+ \left(\rho_1h_1 + \rho_2h_2+\rho_3h_3\right) \dot w^2\right]~dx,&\\
 \mb P+\mb E = \frac{1}{2} \int_0^L \left[\alpha^3 h_3  \left( (v_x^3)^2 +\frac{h_3^2}{12}w_{xx}^2\right) -2\gamma\beta h_3  v^3_x p_x + \beta h_3 p_x^2 \right.&\\
  \quad + \alpha^2 h_2 \left((v^2_x)^2 + \frac{h_2^2}{12}\psi_{x}^2  + G_2h_2 \left(\phi^2\right)^2+ \alpha^1 h_1 \left( (v_x^1)^2 +\frac{h_1^2}{12}w_{xx}^2\right)\right]~dx,\quad\quad&\\
  \mb{B}= \frac{\mu h}{2}\int_{0}^L \dot p^2~dx,&\\
 \mb{W} =\int_0^L \left[- p_x V(t) \right]~dx+  g^1 v^1 (L)  + g^3 v^3 (L) +  g w(L)-Mw_x(L).
 \end{array}
\end{eqnarray}
Here, $M=m_1+m_3,$ $p=\int_0^{x}  D_3(\xi,  t) ~d\xi$ is the total electric charge at point $x,$ $\rho_i$ is the volume density of the $i^{\rm th}$ layer,  and  $\mb K,$ $\mb P+\mb E,$ $\mb B,$ and $\mb W$ are the kinetic energy, the total stored energy, and the magnetic energy of the beam, and the work done by the external mechanical and electrical forces \cite{Ozkan3}.

\subsection{Hamilton's Principle}

By using (\ref{defs1})-(\ref{defs3}), the variables $v^2, \phi^2$  and $\psi^2$ can be written as the functions of state variables as the following
\begin{eqnarray}
\nonumber v^2&=&\frac{1}{2}\left(v^1+v^3\right) + \frac{h_3-h_1}{4}w_x, ~~ \psi^2 =\frac{1}{h_2}\left(-v^1+v^3\right) + \frac{h_1 +  h_3}{2h_2}w_x\\
\nonumber \phi^2 &=&\frac{1}{h_2}\left(-v^1+v^3\right) + \frac{h_1 + 2h_2+ h_3}{2h_2}w_x.
\end{eqnarray}
 Thus, we choose $w,v^1,v^3$ as the state variables.  Let $H=\frac{h_1 + 2h_2+h_3}{2}.$ Application of Hamilton's principle, by using forced boundary conditions, i.e. clamped at $x=0,$ by setting the variation of admissible displacements $\{v^1,v^3, p, w\}$ of ${\mb L}$ to zero yields the following coupled equations of stretching in odd layers, dynamic charge in the piezoelectric layer, the bending of the whole composite:
\begin{eqnarray}
 \label{perturbed}
 \left\{ \begin{array}{ll}
  \left(\rho_1h_1 + \frac{\rho_2h_2}{3}\right)\ddot v^1 + \frac{\rho_2 h_2}{6}\ddot v^3- \frac{1}{6}\alpha^2 h_2 v^3_{xx}  -\left(\alpha^1 h_1 + \frac{1}{3}\alpha^2 h_2\right) v^1_{xx}  & \\
  \quad -\frac{\rho_2 h_2}{12}(2h_1-h_3)\ddot w_x + \frac{\alpha^2 h_2}{12}(2h_1-h_3) w_{xxx}- G_2  \phi^2 = f^1,&\\
   \left(\rho_3h_3 + \frac{\rho_2h_2}{3}\right)\ddot v^3 + \frac{\rho_2 h_2}{6}\ddot v^1- \frac{1}{6}\alpha^2 h_2 v^1_{xx} -\left(\alpha^3 h_3 + \frac{1}{3}\alpha^2 h_2\right) v^3_{xx}&\\
  \quad +\frac{\rho_2 h_2}{12}(2h_3-h_1)\ddot w_x   - \frac{\alpha^2 h_2}{12}(2h_3-h_1) w_{xxx}+ G_2 \phi^2  + \gamma \beta h_3 p_{xx} = f^3,   & \\
        \mu  h_3 \ddot p   -\beta h_3   p_{xx} + \gamma \beta h_3 v^3_{xx}= 0, &\\

   (\rho_1 h_1 + \rho_2 h_2+ \rho_3 h_3) \ddot w    - \frac{1}{12}\left(\rho_1 h_1^3 + \rho_3 h_3^3 + \rho_2h_2(h_1^2 + h_3^2 -h_1h_3)\right) \ddot{w}_{xx} &\\
     \quad + \frac{1}{12}\left(\alpha^1 h_1^3+\alpha^3 h_3^3 + \alpha^2 h_2(h_1^2+h_3^2 -h_1h_3) \right)w_{xxxx}  +\frac{\rho_2 h_2}{12} (2h_1-h_3) \ddot v^1_{x}&\\
   \quad -\frac{\alpha^2 h_2}{12}(2h_1-h_3) v^1_{xxx}  -\frac{\rho_2 h_2}{12}(2h_3-h_1) \ddot v^3_{x} +\frac{\alpha^2 h_2}{12} (2h_3-h_1) v^3_{xxx}&\\
\quad  -  H G_2   \phi^2_x=f.&\\
  \end{array}\right.
\end{eqnarray}
with  associated boundary and initial conditions
\begin{eqnarray}
 \label{ivp}
 \left\{ \begin{array}{ll}
  \left.v^1, v^3, w, w_x, p~~\right|_{x=0}=0,& \\
   \left.\frac{1}{6} \alpha^2 h_2 v^3_x + \left(\alpha^1 h_1 + \frac{1}{3} \alpha^2 h_2\right) v^1_x  \right. &\\
  \left. \quad +\frac{\rho_2h_2}{12}(2h_1-h_3) \ddot w~-\frac{\alpha^2 h_2}{12} (2h_1 - h_3)w_{xx}\right|_{x=L}=g^1(t),&\\
  \left. \frac{1}{6} \alpha^2 h_2 v^1_x + \left(\alpha^3 h_3 + \frac{1}{3} \alpha^2 h_2\right) v^3_x -\frac{\rho_2h_2}{12}(2h_3-h_1) \ddot w\right. &\\
     \left. \quad +\frac{\alpha^2 h_2}{12} (2h_3 - h_1)w_{xx}-\gamma\beta h_3 p_x\right|_{x=L}=g^3(t),&\\
   \left.\beta h_3 p_x -\gamma \beta h_3 v^3_x =-V(t)\right|_{x=L},&\\
  \left.\frac{\alpha^1 (h_1)^3 + \alpha_3 (h_3)^3 + \alpha_2 h_2 ((h_1)^2 + (h_3)^2 -h_1 h_3)}{12} w_{xx}-\frac{\alpha_2 h_2}{12} (2h_1-h_3) v^1_{x}\right. &\\

   \left.\quad +\frac{\alpha_2 h_2}{12} (2h_3-h_1) v^3_{x}  \right|_{x=L} =M(t),&\\
   \frac{\rho_1 (h_1)^3 + \rho_3 (h_3)^3 + \rho_2 h_2 ((h_1)^2 + (h_3)^2 -h_1 h_3)}{12} \ddot w  -\frac{\rho_2 h_2}{12} (2h_1-h_3) \ddot v^1& \\
   \quad +\frac{\alpha_2 h_2}{12} (2h_1-h_3) v^1_{xx}  -\frac{\alpha^1 (h_1)^3 + \alpha_3 (h_3)^3 + \alpha_2 h_2 ((h_1)^2 + (h_3)^2 -h_1 h_3)}{12} w_{xxx}&\\
\left.  \quad +\frac{\rho_2 h_2}{12} (2h_3-h_1) \ddot v^1 -\frac{\alpha_2 h_2}{12} (2h_3-h_1) v^3_{xx} +G_2 H \phi^2 \right|_{x=L} =g(t),&\\
 (v^1, v^3, p, w, \dot v^1, \dot v^3, \dot p, \dot w)(x,0)=(v^1_0,  v^3_0, w_0, p_0,  v^1_1, v^3_1,  p_1, w_1).
  \end{array}\right.
\end{eqnarray}
Note that the equations of motion (\ref{perturbed})-(\ref{ivp}) does not have any distributed damping term. It is simply because the aim of the paper is to investigate the stabilizability of the closed-loop system with only the $B^*-$type stabilizing boundary controllers. Presumably, adding a (viscous) distributed damping term in the form of a shear damping or Kelvin-Voight damping would automatically make the structure asymptotically stable.  For practical applications, it is more relevant to consider a shear-type of damping due to the viscoelastic middle layer by replacing the term $G_2\phi^2$ by $G_2 \phi^2 + \tilde G_2\dot \phi^2$ in (\ref{perturbed}) -(\ref{ivp}) where $\tilde G_2$ is the damping coefficient \cite{Hansen3}.

\section{Fully-dynamic Rao-Nakra (R-N) model}
\label{modeling}
The model obtained above is highly coupled, and it is not very easy to analyze the controllability properties. For this reason, we assume the thin-compliant-layer Rao-Nakra sandwich beam assumptions that the viscoelastic layer is thin and its stiffness negligible. Therefore we work with the perturbed model for $\rho_2, \alpha^2\to 0;$ as in  \cite{Hansen3}. This approximation retains the potential energy of shear and transverse kinetic energy so that the model above reduces to
\begin{eqnarray}
 \label{dbas} \left\{
  \begin{array}{ll}
 \rho_1h_1\ddot v^1-\alpha^1 h_1 v^1_{xx} - G_2 \phi^2 = f^1,   & \\
 \rho_3h_3\ddot v^3-\alpha^3 h_3v^3_{xx} + \gamma \beta h_3 p_{xx} + G_2 \phi^2 = f^3,   & \\
 \mu  h_3 \ddot p   -\beta h_3   p_{xx} + \gamma \beta h_3 v^3_{xx}= 0, &\\
  m \ddot w - K_1 \ddot{w}_{xx} + K_2 w_{xxxx} - G_2 H \phi^2_x=f,&\\
 \phi^2=\frac{1}{h_2}\left(-v^1+v^3 + H w_x\right)&  \end{array} \right.
\end{eqnarray}
with the boundary and initial conditions
\begin{eqnarray}
 \label{d-son} \left\{
  \begin{array}{ll}
 v^1(0)=v^3(0)= p(0)=w(0)=w_x(0)=0,~~ \alpha^1 h_1 v^1_x(L)=g^1(t), &\\
 \alpha^3 h_3 v^3_{x}(L)-\gamma \beta h_3 p_x(L)=g^3(t),~~ \beta h_3 p_x(L) -\gamma \beta h_3v^3_x(L)= -V(t),\\
  K_2 w_{xx}(L) = -M(t), K_1 \ddot w_x(L) -K_2 w_{xxx}(L) + G_2 H \phi^2(L)=g(t),&\\
(v^1, v^3, p, w, \dot v^1, \dot v^3, \dot p, \dot w)(x,0)=(v^1_0,  v^3_0, p_0, w_0, v^1_1, v^3_1,  p_1, w_1)&
 \end{array} \right.&&
\end{eqnarray}
where $m=\rho_1h_1+ \rho_2h_2 + \rho_3 h_3,$ $K_1=\frac{\rho_1 h_1^3}{12} +\frac{\rho_3 h_3^3}{12},$ and $K_2=\frac{\alpha^1 h_1^3}{12}+\frac{\alpha^3 h_3^3}{12}.$
\vspace{0.1in}

\noindent {\bf{\large Semigroup well-posedness:}} \label{Sec-III}
Define
\begin{eqnarray}
 \nonumber
  \begin{array}{ll}
 H^1_L(0,L)=\{\psi\in H^1(0,L): \psi(0)=0\}, &\\
   H^2_L(0,L)=\{\psi\in H^2(0,L): \psi(0)=\psi_x(0)=0\},   &
   \end{array}
\end{eqnarray}
and  the complex linear spaces
\begin{eqnarray}
 \nonumber
  \begin{array}{ll}
 \mX=\Ltwo(0,L),~~  \mathrm V=\left(H^1_L(0,L)\right)^3 \times H^2_L(0,L),   \mathrm H= \mX^3 \times H^1_L(0,L),~~ \mc{H} = \mathrm V \times \mathrm H.  &
   \end{array}
\end{eqnarray}
The  energy associated with (\ref{dbas})-(\ref{d-son})  is
\begin{eqnarray}
\nonumber
  \begin{array}{ll}
 \mathrm{E}(t) =\frac{1}{2}\int_0^L \left\{\rho_1 h_1  |\dot v^1|^2 + \rho_3 h_3  |\dot v^3|^2 + \mu h_3  |\dot p|^2+ m |\dot w|^2 + \alpha^1 h_1 |v^1_x|^2 + \alpha^3 h_3 |v^3_x|^2\right.  & \\
 \left.  + K_1 |\dot w_x|^2 + K_2 |w_{xx}|^2 -\gamma\beta h_3 v_x \bar p_x-\gamma\beta h_3  p_x \bar v^3_x + \beta h_3 |p_x|^2 +G_2 h_2 |\phi^2|^2 \right\}~ dx.~~&\\
   \end{array}
\end{eqnarray}
This motivates the  definition of the inner product on $\mc H:$
\begin{eqnarray}
\label{oz} \begin{array}{ll}
 \left<\left[ \begin{array}{c}
 u_1 \\
 \vdots\\
 u_8
 \end{array} \right], \left[ \begin{array}{c}
 v_1 \\
 \vdots\\
 v_8
 \end{array} \right]\right>_{\mc H}= \left<\left[ \begin{array}{l}
 u_5\\
 u_6\\
 u_7\\
 u_8
 \end{array} \right], \left[ \begin{array}{l}
 v_5\\
 v_6\\
 v_7\\
 v_8
 \end{array} \right]\right>_{\mathrm H}  + \left<\left[ \begin{array}{l}
 u_1 \\
 u_2 \\
 u_3\\
 u_4
 \end{array} \right], \left[ \begin{array}{l}
 v_1 \\
 v_2\\
 v_3\\
 v_4
 \end{array} \right]\right>_{\mathrm V}&\\
 =\int_0^L \left\{\rho_1 h_1  u_5 { {\bar v}}_5 + \rho_3 h_3  u_6 { {\bar v}}_6+ \mu h_3  u_7 { {\bar v}_7} + m  u_8{ {\bar v}_8} + K_1 (u_8)_x(\bar v_8)_x  \right.& \\
 \quad  + \alpha^1 h_1  (u_1)_{x} (\bar v_1)_x + K_2 (u_4)_{xx} (\bar v_4)_{xx} &\\
  \quad + \frac{G_2}{h_2}  (-u_1+u_2 + H(u_4)_x)(-\bar v_1+\bar v_2 + H(\bar v_4)_x) &\\
 \left.\quad + h_3 \left< \left[ {\begin{array}{*{20}c}
   \alpha^3 + \gamma^2\beta  & -\gamma\beta\\
     -\gamma \beta & \beta  \\
\end{array}} \right] \left[ \begin{array}{l}
 u_{2x} \\
 u_{3x}
 \end{array} \right], \left[ \begin{array}{l}
  v_{2x} \\
  v_{3x}  \end{array} \right]\right>_{\mathbb{C}^2} \right\}~dx
  \end{array}
 \end{eqnarray}
  where $\left<\cdot,\cdot\right>_{\mathbb{C}^2}$ is the inner product on $\mathbb{C}^2.$  Obviously, $\langle \, , \, \rangle_{\mc H} $ does indeed define an inner product, with the induced energy norm,  since  $ \left| {\begin{array}{*{20}c}
   \alpha^3 + \gamma^2\beta  & -\gamma\beta\\
     -\gamma \beta & \beta  \\
\end{array}} \right|>0.$ As well, $\|-u_1+u_2 + H(u_4)_x\|_{L^2(0,L)}$ does not violate the coercivity of  (\ref{oz}), see (Lemma 2.1, \cite{Hansen3}) for the details.

Consider that all external body forces are zero, i.e. $f^1,f^3,f\equiv 0.$  Writing $\varphi=[v^1, v^3, p, w, \dot v^1,  \dot v^3, \dot p, \dot w]^{\rm T}$ and ${\bf F}(t)=\left(g^1(t), g^3(t), V(t), M(t), g(t) \right),$ the control  system (\ref{dbas})-(\ref{d-son})   can be put into the  state-space form
\begin{eqnarray}\label{Semigroup}\left\{
\begin{array}{ll}
\dot \varphi +  \underbrace{ \left[ {\begin{array}{*{20}c}
   0_{4\times 4}  & I_{4\times 4} \\
       M^{-1}A  &  0_{4\times 4}  \\
\end{array}} \right]}_{\mc A}  \varphi =\underbrace{ \left[ \begin{array}{c} 0_{4\times 5} \\ B_0 \end{array} \right]}_B {\bf F}(t),\quad \varphi(x,0) =  \varphi ^0.
\end{array}\right.
\end{eqnarray}
where $\mc A: {\text{Dom}}(\mc A) \subset \mc H \to \mc H$ with ${\rm {Dom}}(\mc A) =  \{(\vec z, \vec {\tilde z}) \in V\times V, A\vec z \in \mathrm V' \},$  and the operators $A:V\to V',$ $\mc M : H^1_L (0,L) \to (H^1_L(0,L))',$ $B \in \mathcal{L}(\mathbb{C}^5 , ({\rm Dom}(\mc A))'), B_0 \in \mathcal{L}(\mathbb{C}^5, \mathrm V')$ are
\begin{eqnarray}
\nonumber &&\left<A \psi, \tilde\psi\right>_{V',V}= \left<\psi, \tilde \psi\right>_{V},\quad M=\left[\rho_1 h_1 I ~~\rho_3 h_3 I~~\mu h_3 I~~ \mc M\right],\\
\nonumber&& \begin{array}{ll}\left< \mc M \psi, \tilde \psi \right>_{(H^1(0,L))', H^1_L(0,L))}= \int_0^L \left(m \psi \tilde \psi + K_1 \psi_x \tilde \psi_x\right) dx,\end{array}\\
 \nonumber &&\begin{array}{ll}
B_0 =   \left(
                                                                                      \begin{array}{ccccc}
                                                                                        \delta_L & 0 & 0 & 0 & 0 \\
                                                                                        0 & \delta_L & 0 & 0 & 0 \\
                                                                                        0 & 0 & -\delta_L & 0 & 0 \\
                                                                                        0 & 0 & 0 & (\delta_L)_x & \delta_L \\
                                                                                      \end{array}
                                                                                    \right).  &
   \end{array}
\end{eqnarray}
 In the above, $\delta_L=\delta(x-L) $ is the Dirac-Delta distribution, and $\mathrm V'$ and $({\rm Dom}(\mc A))'$  are dual spaces of $\mathrm V$ and ${\rm Dom}(\mc A)$ pivoted with respect $\mathrm H,$ respectively.
We have the following well-posedness and perturbation  theorem:

\begin{theorem}\label{w-pf}
Let $T>0,$ and ${\bf F}\in (\Ltwo(0,T))^5.$ For any $\varphi^0 \in \mc H,$ $\varphi\in C[[0,T]; \mc H],$ and there exists a positive constants $c_1(T)$
such that (\ref{Semigroup}) satisfies
      \begin{eqnarray} \nonumber
      \begin{array}{ll}
      \|\varphi (T) \|^2_{\mc H} \le c_1 (T)\left(\|\varphi^0\|^2_{\mc H} + \|{\bf F}\|^2_{(\Ltwo(0,T))^5}\right).
      \end{array}
      \end{eqnarray}
\end{theorem}
\vspace{0.1in}
\begin{proof} The proof is provided in \cite{Ozkan5}.
\end{proof}

\begin{theorem} For fixed initial data and no applied forces, the solution $(v^1,v^3,p,w)\in \mc H$ of (\ref{perturbed})-(\ref{ivp}) converges to  the solution of $(v^1,v^3,p,w)\in \mc H$ in (\ref{dbas})-(\ref{d-son})  as $\|(\rho_2,\alpha^2)\|_{\mathbb{R}^2}\to 0.$
\end{theorem}
\begin{proof} The system (\ref{dbas})-(\ref{d-son}) is the perturbation of the system (\ref{perturbed})-(\ref{ivp}). The proof is analogous to the one in  \cite{Hansen3}.
\end{proof}

\subsection{Electrostatic Rao-Nakra (R-N) model }
\label{Nomag-RN}
Note that, if we exclude the magnetic effects by $\mu \to 0$ in (\ref{dbas})-(\ref{d-son}), or $\mb B\equiv 0$ in (\ref{energies}),  the $p-$equation in (\ref{dbas})-(\ref{d-son}) can be solved for $p.$ Then, the simplified equations of motion are as the following
\begin{eqnarray}
 \label{d4-non} &&\left\{
  \begin{array}{ll}
   \rho_1h_1\ddot v^1-\alpha^1h_1 v^1_{xx} - G_2  \phi^2 = 0   & \\
    \rho_3h_3\ddot v^3-\alpha_{1}^3 h_3 v^3_{xx} + G_2  \phi^2 = 0   &\\
   m \ddot w - K_1 \ddot{w}_{xx} +K_2  w_{xxxx} - {G_2 H}  \phi^2_x= 0&\\
    \phi^2=\frac{1}{h_2}\left(-v^1+v^3 + H w_x\right)&
  \end{array} \right.
\end{eqnarray}
with the boundary and initial conditions
\begin{eqnarray}
\label{divp-non} \left\{
  \begin{array}{ll}
 v^1(0)=v^3(0)= w(0)=w_x(0)=0, & \\
\alpha_1 h_1 v^1_x(L)= g^1(t), \quad \alpha_1^3 h_3 v^3_{x}(L)= -\gamma V(t),& \\
  K_2 w_{xx}(L) = -M(t),\quad
  K_1 \ddot w_x(L) -K_2 w_{xxx}(L) + G_2 H \phi^2(L)=g(t), &\\
  (v^1, v^3, w, \dot v^1, \dot v^3, \dot w)(x,0)=(v^1_0,  v^3_0, w_0,  v^1_1, v^3_1,  w_1).
  \end{array} \right.
\end{eqnarray}
where we use (\ref{coef}). We also free $g^3(t)\equiv 0,$ which is the mechanical strain controller at the tip $x=L,$ since the voltage control $V(t)$ of the piezoelectric layer is able to control the strains by itself.

\noindent \textbf{Semigroup Formulation:} Define the complex linear space $\mc{H} = \mathrm V \times \mathrm H$ where
$ V=\left(H^1_L(0,L)\right)^2 \times H^2_L(0,L),~~ \mathrm H= \mX^2 \times H^1_L(0,L).$ The  natural energy associated with (\ref{d4-non})-(\ref{divp-non})  is
\begin{eqnarray}
\nonumber
  \begin{array}{ll}
\mathrm{E}(t) =\frac{1}{2}\int_0^L \left\{\rho_1 h_1  |\dot v^1|^2 + \rho_3 h_3  |\dot v^3|^2 + m |\dot w|^2  + \alpha^1 h_1 |v^1_x|^2  + \alpha^3_1 h_3 |v^3_x|^2+ K_1 |\dot w_x|^2\right. &\\
   \left.\quad\quad \quad + K_2 |w_{xx}|^2 +G_2 h_2 |\phi^2|^2 \right\}~ dx.
  \end{array}
\end{eqnarray}
 This motivates definition of the inner product on $\mc H$
 \begin{eqnarray}
  \nonumber \begin{array}{ll}
\left<\left[ \begin{array}{c}
 u_1 \\
\vdots\\
 u_6
 \end{array} \right], \left[ \begin{array}{c}
 v_1 \\
 \vdots\\
 v_6
 \end{array} \right]\right>_{\mc H}= \left<\left[ \begin{array}{l}
 u_4\\
 u_5\\
 u_6
 \end{array} \right], \left[ \begin{array}{l}
 v_4\\
 v_5\\
 v_6
 \end{array} \right]\right>_{\mathrm H}
   + \left<\left[ \begin{array}{l}
 u_1 \\
 u_2 \\
 u_3
 \end{array} \right], \left[ \begin{array}{l}
 v_1 \\
 v_2\\
 v_3
 \end{array} \right]\right>_{\mathrm V}&\\
 =\int_0^L \left\{\rho_1 h_1  u_5 {\dot {\bar v}}_5 + \rho_3 h_3  u_6 {\dot {\bar v}}_6 + \mu h_3  \dot u_7 {\dot {\bar v}_7} + m \dot u_8{\dot {\bar v}_8}+ K_1 (u_8)_x(\bar v_8)_x \right. &\\
  \quad\quad + \alpha^1 h_1  (u_1)_{x} (\bar v_1)_x + K_2 (u_4)_{xx} (\bar v_4)_{xx} + \alpha^1_3 h_3  (u_3)_{x} (\bar v_3)_x&\\
 \left. \quad\quad +\frac{G_2}{h_2}  (-u_1+u_2 + H(u_4)_x)(-\bar v_1+\bar v_2 + H(\bar v_4)_x)\right\}~dx.
  \end{array}
\end{eqnarray}
 Writing $\varphi=[v^1, v^3, w, \dot v^1,  \dot v^3,  \dot w]^{\rm T}$ and ${\bf F}(t)=\left(g^1(t), V(t), M(t), g(t) \right),$ the control  system (\ref{d4-non})-(\ref{divp-non}) can be put into the  state-space form
\begin{eqnarray}
\label{Semigroup-non}\left\{
\begin{array}{ll}
\dot \varphi + \underbrace{\left[ {\begin{array}{*{20}c}
   0_{3\times 3}  & I_{3\times 3} \\
       M^{-1}A  &  0_{3\times 3}  \\
\end{array}} \right]}_{\mc A}  \varphi =\underbrace{ \left[ \begin{array}{c} 0_{3\times 4} \\ B_0 \end{array} \right]}_B {\bf F}(t) , \quad \varphi(x,0) =  \varphi ^0.
\end{array}\right.
\end{eqnarray}
where ${\rm {Dom}}(\mc A) =  \{(\vec z, \vec {\tilde z}) \in V\times V, A\vec z \in \mathrm V' \},$ $\mc A: {\text{Dom}}(A)\times V\subset \mc H \to \mc H,$ $M=\left[\rho_1 h_1 I ~~\rho_3 h_3 I~~ \mc M\right],$ and the operators $B \in \mathcal{L}(\mathbb{C}^4 , ({\rm Dom}(\mc A))'), B_0 \in \mathcal{L}(\mathbb{C}^4, \mathrm V')$ are
\begin{eqnarray}
\nonumber &&\left<A \psi, \tilde\psi\right>_{V',V}= \left<\psi, \tilde \psi\right>_{V},\quad\begin{array}{ll}
B_0 =  \left[
                                                                                      \begin{array}{ccccc}
                                                                                        \delta_L & 0 & 0 & 0  \\
                                                                                        0 & \delta_L & 0 & 0  \\
                                                                                        0 & 0 & (\delta_L)_x & \delta_L \\
                                                                                      \end{array}
                                                                                    \right].  &
   \end{array}
\end{eqnarray}
\begin{theorem}\label{w-pf-non}
Let $T>0,$ and ${\bf F}\in (\Ltwo(0,T))^4\in \Ltwo(0,T).$ For any $\varphi^0 \in \mc H,$ $\varphi\in C[[0,T]; \mc H]$ and  there exists a positive constants $c_1(T)$
such that (\ref{Semigroup-non}) satisfies
      \begin{eqnarray}\nonumber
      \begin{array}{ll}
      \|\varphi (T) \|^2_{\mc H} \le c_1 (T)\left\{\|\varphi^0\|^2_{\mc H}+ \|{\bf F}\|^2_{(\Ltwo(0,T))^4}\right\}.
      \end{array}
      \end{eqnarray}
\end{theorem}
\begin{proof} The proof can be found in \cite{Ozkan5}.
\end{proof}

\begin{theorem} For fixed initial data and no applied forces, the solution $(v^1,v^3,w)\in \mc H$ of (\ref{perturbed})-(\ref{ivp}) converges to  the solution of $(v^1,v^3,w)\in \mc H$ in (\ref{d4-non})-(\ref{divp-non}) as \\$\|(\rho_2,\alpha^2,\mu)\|_{\mathbb{R}^3}\to 0.$
\end{theorem}

\section{ Fully dynamic Mead-Marcus (M-M) model}
Another way to simplify (\ref{perturbed})-(\ref{ivp}) is to assume the Mead-Marcus type sandwich beam assumptions that the longitudinal and rotational inertia terms are negligible in (\ref{perturbed})-(\ref{ivp}). This type of perturbation is singular, and therefore,   solutions do not behave continuously with respect to this perturbation. However, it is noted in \cite{Hansen3} that this type of approximation provides a close approximation to
the original  Rao-Nakra model in the low-frequency range. In this section,  we follow the  methodology in \cite{F}. By $\ddot v^1, \ddot v^3, \ddot w_{xx}\to 0$ in (\ref{perturbed})-(\ref{ivp}), we obtain
\begin{eqnarray}
 \label{perturbeddd} \left\{
  \begin{array}{ll}
   -\frac{1}{6}\alpha^2 h_2 v^3_{xx}-\left(\alpha^1 h_1 + \frac{1}{3}\alpha^2 h_2\right) v^1_{xx}     + \frac{\alpha^2 h_2}{12}(2h_1-h_3) w_{xxx}- G_2  \phi^2 = f^1, &\\
   -\frac{1}{6}\alpha^2 h_2 v^1_{xx}-\left(\alpha^3 h_3 + \frac{1}{3}\alpha^2 h_2\right) v^3_{xx}  - \frac{\alpha^2 h_2}{12}(2h_3-h_1) w_{xxx}&\\
   \quad\quad\quad+ G_2 \phi^2  + \gamma \beta h_3 p_{xx}  = f^3,   & \\
        \mu  h_3 \ddot p   -\beta h_3   p_{xx} + \gamma \beta h_3 v^3_{xx}= -V(t) \delta_L, &\\
      m \ddot w -\frac{\alpha^2 h_2}{12}(2h_1-h_3) v^1_{xxx}+\frac{\alpha^2 h_2}{12} (2h_3-h_1) v^3_{xxx}   &\\
 \quad + \frac{1}{12}\left(\alpha^1 h_1^3+\alpha^3 h_3^3 + \alpha^2 h_2(h_1^2+h_3^2 -h_1h_3) \right)w_{xxxx} -  H G_2   \phi^2_x=f,&\\
      \phi^2=\frac{1}{h_2}\left(-v^1+v^3 + H w_x\right),&
 \end{array} \right.
\end{eqnarray}
with the boundary and initial conditions
\begin{eqnarray}
 \label{ivpd} \left\{
  \begin{array}{ll}
 \left.v^1, v^3, w, w_x, p~~\right|_{x=0}=0, \\
   \left.\frac{1}{6} \alpha^2 h_2 v^3_x + \left(\alpha^1 h_1 + \frac{1}{3} \alpha^2 h_2\right) v^1_x -\frac{\alpha^2 h_2}{12} (2h_1 - h_3)w_{xx}\right|_{x=L}=g^1(t), \\
  \frac{1}{6} \alpha^2 h_2 v^1_x + \left(\alpha^3 h_3 + \frac{1}{3} \alpha^2 h_2\right) v^3_x  +\frac{\alpha^2 h_2}{12} (2h_3 - h_1)w_{xx} &\\
  \left. \quad\quad\quad\quad- \gamma \beta h_3 p_{x}\right|_{x=L}=g^3(t), \\
\left.\beta h_3 p_x(L) -\gamma \beta h_3v^3_x(L)\right|_{x=L}= 0, \\
  \left.\frac{\alpha^1 (h_1)^3 + \alpha_3 (h_3)^3 + \alpha_2 h_2 ((h_1)^2 + (h_3)^2 -h_1 h_3)}{12} w_{xx}\right.  \\
 \left. \quad  -\frac{\alpha_2 h_2}{12} (2h_1-h_3) v^1_{x}+\frac{\alpha_2 h_2}{12} (2h_3-h_1) v^3_{x} \right|_{x=L} =M(t),\\
  \left.\quad -\frac{\alpha^1 (h_1)^3 + \alpha_3 (h_3)^3 + \alpha_2 h_2 ((h_1)^2 + (h_3)^2 -h_1 h_3)}{12} w_{xxx}\right.  \\
 \left.  \quad+\frac{\alpha_2 h_2}{12} (2h_1-h_3) v^1_{xx} -\frac{\alpha_2 h_2}{12} (2h_3-h_1) v^3_{xx} +G_2 H \phi^2 \right|_{x=L} =g(t),\\
(v^1, v^3,  p, w, \dot v^1, \dot v^3, \dot p, \dot w)(x,0)=(v^1_0,  v^3_0, p_0, w_0,  v^1_1, v^3_1, p_1,  w_1).
 \end{array} \right. &&
\end{eqnarray}
For the rest of the section, we free all mechanical controllers $g^1(t), g^3(t), M(t), g(t)\equiv 0$ and just keep voltage controller $V(t).$
Let $\varsigma=\frac{G_2}{h_2},$ and define the positive constants by
\begin{eqnarray}
\label{constants} \begin{array}{ll}
 A= \frac{1}{12}\left(\alpha^1h_1^3 + \alpha^3 h_3^3+ \frac{\alpha^2 h_2^2\left(3\alpha^2 h_2(\alpha^1h_1^3 + \alpha^3 h_3^3)+12\alpha^1\alpha^3h_1 h_3 (h_1^2+h_3^2-h_1h_3)\right)}{\alpha^2 (h_2)^2 \left(4\alpha^1 h_1 + \alpha^2 h_2 + 4\alpha^3 h_3\right)+12 \alpha^1\alpha^3 h_1h_2 h_3}\right),&\\
 B_1= \frac{\alpha^2\left[ \alpha^2 h_2^2 + 3 \alpha^1 h_1^2 +4\alpha^1 h_1 h_2 + 3 \alpha^3 h_3^2 + 4 \alpha^3 h_2 h_3\right]+ 12 \alpha^1\alpha^3 H h_1 h_3}{\alpha^2 h_2^2 \left(4\alpha^1 h_1 + \alpha^2 h_2 + 4\alpha^3 h_3\right)+12 \alpha^1\alpha^3 h_1 h_2h_3},&\\
  B_2= \frac{6\alpha^2 h_2 +12\alpha^1h_1}{\alpha^2 h_2^2 \left(4\alpha^1 h_1 + \alpha^2 h_2 + 4\alpha^3 h_3\right)+12 \alpha^1\alpha^3 h_1 h_2h_3},&\\
B_3=\frac{ \frac{1}{2}(\alpha^2)^2h_2^3h_3^2 + 2\alpha^1\alpha^2h_1h_2^2h_3^2-\alpha^1\alpha^2h_1^2h_2^2h_3}{\alpha^2 h_2^2 \left(4\alpha^1 h_1 + \alpha^2 h_2 + 4\alpha^3 h_3\right)+12 \alpha^1\alpha^3 h_1 h_2h_3},&\\
  B_4=\frac{(\alpha^2)^2 h_2^3 h_3 + 12 \alpha^1 \alpha^3_1 h_1 h_2 h_3^2+ 4\alpha^1\alpha^2 h_1h_2^2h_3 + 4\alpha^2\alpha^3_1 h_2^2h_3^2}{\alpha^2 h_2^2 \left(4\alpha^1 h_1 + \alpha^2 h_2 + 4\alpha^3 h_3\right)+12 \alpha^1\alpha^3_1 h_1 h_2h_3},&\\
 C=\frac{12\left(\alpha^1h_1 +\alpha^2 h_2 +\alpha^3 h_3\right)}{\alpha^2 h_2 \left(4\alpha^1 h_1 + \alpha^2 h_2 + 4\alpha^3 h_3\right)+12 \alpha^1\alpha^3 h_1 h_3}.
\end{array}
\end{eqnarray}
Now we multiply the  first and second equations in (\ref{perturbeddd}) by $\left(\frac{1}{24}\alpha^2  h_2 + \frac{1}{12} \alpha^3 h_3\right) $ and $-\left(\frac{1}{24}\alpha^2  h_2 + \frac{1}{12} \alpha^1 h_1\right),$ respectively, and add these two equations. An alternate formulation is obtained as the following
\begin{eqnarray}
\label{perturbed-dumb} \left\{ \begin{array}{ll}
  m \ddot w + A w_{xxxx} -  B_1\varsigma \gamma\beta h_2 h_3 \phi^2_x + \gamma\beta B_3p_{xxx}=0,&\\
  C\varsigma \phi^2 -\phi^2_{xx} + B_1w_{xxx} + B_2 p_{xx}=0,&\\
 \mu  h_3 \ddot p   - \beta B_4  p_{xx} + \gamma \beta h_2h_3 \varsigma B_2\phi^2 -\gamma\beta B_3 w_{xxx}  = -V(t)\delta_L&
  \end{array} \right.
\end{eqnarray}
with the boundary and initial conditions
\begin{eqnarray}
\nonumber \begin{array}{ll}
  \left| w, w_x, \phi^2, p~\right|_{x=0}=0,&\\
  \left| A w_{xx} + \gamma\beta B_3p_{x}\right|_{x=L}=0,\quad
  \left|  \beta B_4  p_{x}  +\gamma\beta B_3 w_{xx} \right|_{x=L}=0,&\\
   \left|- A w_{xxx} + B_1\varsigma \gamma\beta h_2 h_3 \phi^2 - \gamma\beta B_3p_{xx}\right|_{x=L}=0,&\\
      \left|\phi^2_{x} -B_1w_{xx} -B_2 p_{x}\right|_{x=L}=0, &\\
( w, p, \dot w, \dot p)(x,0)=(w_0,  p_0, w^1, p^1).
\end{array}
\end{eqnarray}
By \ref{constants}, the boundary conditions can be further simplified to obtain
\begin{eqnarray}
\nonumber && \left| w, w_x, \phi^2, p~\right|_{x=0}=0, \quad \left| w_{xx}= \phi^2_x=p_x\right|_{x=L}=0, \\
\nonumber &&\left| - A w_{xxx} + B_1\varsigma \gamma \beta h_2 h_3 \phi^2 - \gamma\beta B_3p_{xx}\right|_{x=L}=0.
\end{eqnarray}
Now let $\xi=C\varsigma.$ The elliptic equation for $\phi^2$ in (\ref{perturbed-dumb}) can be written as
\begin{eqnarray}
\label{elliptic}\left\{
\begin{array}{ll}- \phi^2_{xx} + \xi \phi^2 = -B_1w_{xxx}-B_2 p_{xx},&\\
 \phi^2(0)=\phi^2_x(L)=0.
\end{array}\right.
\end{eqnarray}
Since $a(\phi, \psi)=\int_0^L \left(\xi ~ \phi_x ~ \psi_x + \phi ~\psi\right)$ is a continuous and coercive bilinear form on $H^1_L(0,L)$, and $b(\psi)=\int_0^L g  \psi~dx$ is a continuous functional on $H^1_L(0,L)$ for given $g\in L^2(0,L),$ by the Lax-Milgram theorem, the elliptic equation $a(\phi,\psi)=b(g,\psi)$ has a unique solution and this solution is $\phi=P_\xi g$ where the operator $P_\xi^{-1}$ is defined by
\begin{eqnarray}\label{Lgamma}P_{\xi}^{-1}:=(- D_x^2+ \xi I)^{-1} .\end{eqnarray}
It follows from (\ref{perturbed-dumb}) that
\begin{eqnarray}\label{solutio} \phi^2= -P_{\xi} \left(B_1 w_{xxx}+B_2 p_{xx}\right).
 \end{eqnarray}
Plugging (\ref{solutio}) in  (\ref{perturbed-dumb}), we obtain
\begin{small}
\begin{eqnarray}
 \nonumber \left\{
  \begin{array}{ll}
   m \ddot w + A w_{xxxx} +(B_1^2 \gamma\beta h_2h_3\varsigma P_\xi w_{xxx})_x + \gamma\beta h_2 h_3(B_1 B_2 \varsigma P_\xi p_{xx})_x+ \gamma\beta B_3p_{xxx}=0,\\
  \mu  h_3 \ddot p   -\beta B_4  p_{xx}-\gamma\beta B_3 w_{xxx}  -\gamma\beta h_2h_3  \varsigma P_\xi \left(B_1B_2 w_{xxx} +B_2^2 p_{xx}\right)   = -V(t)\delta_L.
  \end{array} \right.
\end{eqnarray}
\end{small}
with the boundary and initial conditions
\begin{eqnarray}
\nonumber \left\{\begin{array}{ll}
 \left| w, w_x, p~\right|_{x=0}=0, \quad \left| w_{xx} , p_x\right|_{x=L}=0,& \\
\left| A w_{xxx} +B_1^2 \gamma \beta h_2 h_3 \varsigma (P_\xi w_{xxx}) +B_1 B_2 \varsigma (P_\xi p_{xx})+\gamma\beta B_3p_{xx}\right|_{x=L}=0,\\
( w, p, \dot w, \dot p)(x,0)=(w_0,  p_0, w^1, p^1).&
\end{array}\right.
\end{eqnarray}
\begin{remark} This model describes the coupling between bending of the whole \emph{composite} and magnetic charge of the piezoelectric layer. This can be compared to the model obtained in \cite{O-M1} where the equations of motion describes the coupling between the longitudinal motions on the \emph{single piezoelectric layer} and the magnetic charge equation. This model is novel and is never mathematically analyzed in the literature.
\end{remark}

\begin{lemma} \label{pxi} Let ${\rm Dom} (D_x^2)=\{w\in H^2(0,L) ~:~ w(0)=w_x(L)=0\}.$ Define the operator $J=\xi P_\xi -  I. $ Then $J$ is continuous, self-adjoint and non-positive on $L^2(0,L).$ Moreover, for all $w\in {\rm Dom} (P_\xi ), $
 \begin{eqnarray}J w =  ~P_\xi  D_x^2= (\xi I- D_x^2)^{-1}D_x^2 w.\label{J}\end{eqnarray}
\end{lemma}
\vspace{0.1in}
\begin{proof} Continuity and self-adjointness easily follow from the definition of $J.$ We first prove that $J$ is a non-positive operator. Let $u\in L^2(0,L).$ Then $P_\xi  u=(\xi I- D_x^2)^{-1} u=s$ implies that $s\in {\rm Dom} (D_x^2)$ and $\xi s- s_{xx}=u$
\begin{eqnarray}\nonumber \left<Ju, u\right>_{L^2(0,L)}&=& \left<(\xi P_\xi -I) u, u\right>_{L^2(0,L)}\\
\nonumber &=& \left< \xi s - \xi s +  s_{xx}, \xi s- s_{xx}\right>_{L^2(0,L)}\\
\nonumber &=& \xi \left< s_{xx}, s \right>_{L^2(0,L)} - \left< s_{xx}, s_{xx} \right>_{L^2(0,L)}\\
\nonumber &=& -\xi \|  s_{x} \|^2_{L^2(0,L)} - \|  s_{xx}\|^2_{L^2(0,L)}.
\end{eqnarray}

To prove (\ref{J}), let $P_\xi D_x^2 w:=v.$ Then $w_{xx}=P_\xi^{-1} v.$ Adding and subtracting $\xi w$ on the left hand side yields $\xi w- \xi w+ w_{xx}=P_\xi^{-1} v,$ we obtain $\xi w- P_\xi^{-1} w = P_\xi^{-1} v,$ and therefore $v=-w+\xi P_\xi w=Jw.$
\end{proof}
The equations in (\ref{perturbed-dumb}) can be simplified as
\begin{eqnarray}
 \label{perturbeddd-an} \left\{
  \begin{array}{ll}
  m \ddot w + A w_{xxxx} + \gamma\beta B_3p_{xxx}+\gamma\beta h_2 h_3 \varsigma \left(B_1^2  J w_{x} +B_1 B_2 J p\right)_x=0,&\\
  \mu  h_3 \ddot p   -\beta B_4  p_{xx}   -\gamma\beta B_3 w_{xxx} &\\
  \quad\quad\quad\quad -\gamma\beta h_2h_3 \varsigma \left(B_2^2 J p + B_1B_2 J w_{x} \right)= -V(t)\delta_L.&
    \end{array} \right.
\end{eqnarray}
with the boundary and initial conditions
\begin{eqnarray}
\label{ivpp}
\begin{array}{ll}
    \left| w, w_x, p~\right|_{x=0}=0, \quad \left| w_{xx} , p_x\right|_{x=L}=0, &\\
   \left| A w_{xxx} +\gamma \beta h_2 h_3 B_1^2 \varsigma (J w_{x})   +\gamma\beta h_2 h_3 B_1 B_2 \varsigma (J p)+\gamma\beta B_3p_{xx}  \right|_{x=L}=0,&\\
  ( w, p, \dot w, \dot p)(x,0)=(w_0,  p_0, w^1, p^1).
\end{array}
\end{eqnarray}

\noindent {\bf{Semigroup well-posedness:}} Define $\mc{H} = \mathrm V \times \mathrm H$ where
\begin{eqnarray}
\begin{array}{ll}
  \mathrm V= H^2_L(0,L) \times H^1_L(0,L),~ \mathrm H= \mX^2=(\Ltwo(0,L))^2.
\end{array}
\end{eqnarray}
The  energy associated with (\ref{perturbeddd-an})-(\ref{ivpp}) is
\begin{eqnarray}
\nonumber && \mathrm{E}(t) =\frac{1}{2}\int_0^L \left\{ m |\dot w|^2 + \mu h_3  |\dot p|^2 + A |w_{xx}|^2 + \beta B_4 |p_x|^2 + \gamma \beta B_3 p_x  \bar w_{xx}  \right.\\
\nonumber && \left.+\gamma \beta B_3 \bar p_x  \ w_{xx}-\gamma \beta h_2 h_3 \varsigma (J(B_1 w_x + B_2 p))(B_1 \bar w_x + B_2 \bar p) \right\} ~dx\\
 \nonumber && =\frac{1}{2}\int_0^L \left\{ m |\dot w|^2 + \mu h_3  |\dot p|^2 -\gamma \beta h_2 h_3 \varsigma (J(B_1 w_x + B_2 p))(B_1 \bar w_x + B_2 \bar p)\right.\\
  \nonumber  && \left.   + \left<\left[ \begin{array}{ll}
 A & \gamma\beta B_3\\
 \gamma\beta B_3 & \beta B_4
 \end{array} \right]\left[ \begin{array}{l}
 w_{xx}\\
 p_x
 \end{array} \right], \left[ \begin{array}{l}
 w_{xx}\\
  p_{x}
 \end{array} \right]\right>_{\mathbb{C}^2} \right\} ~dx.\quad\quad\quad
\end{eqnarray}
 This motivates definition of the inner product on $\mc{H}$
\begin{eqnarray}
\nonumber \begin{array}{ll}
 \left<\left[ \begin{array}{l}
 u_1 \\
 u_2 \\
 u_3\\
 u_4
 \end{array} \right], \left[ \begin{array}{l}
 v_1 \\
 v_2 \\
 v_3\\
 v_4
 \end{array} \right]\right>_{\mc{H}}= \left<\left[ \begin{array}{l}
 u_3\\
 u_4
 \end{array} \right], \left[ \begin{array}{l}
 v_3\\
 v_4
 \end{array} \right]\right>_{\mathrm H}+  \left<\left[ \begin{array}{l}
 u_1 \\
 u_2
 \end{array} \right], \left[ \begin{array}{l}
 v_1 \\
 v_2
 \end{array} \right]\right>_{\mathrm V} & \\
 =\int_0^L \left\{mu_3 { {\bar v}_3}+ \mu h_3  u_4 { {\bar v}_4} -\gamma \beta h_2 h_3 \varsigma (J(B_1 (u_1)_x + B_2 u_2))(B_1 (\bar u_1))_x + B_2 (\bar u_2)) \right. &\\
  +\left. \left<\left[ \begin{array}{ll}
 A & \gamma\beta B_3\\
 \gamma\beta B_3 & \beta B_4
 \end{array} \right]\left[ \begin{array}{l}
 (u_1)_{xx}\\
 (u_2)_x
 \end{array} \right], \left[ \begin{array}{l}
 (u_1)_{xx}\\
  (u_2)_{x}
 \end{array} \right]\right>_{\mathbb{C}^2}  \right\}~dx.
 \end{array}
 \end{eqnarray}
Here $\langle \, , \, \rangle_{\mc H} $ defines an inner product with induced energy norm since $J$ is a non-positive operator, and the matrix $\left[ \begin{array}{ll}
 A & \gamma\beta B_3\\
 \gamma\beta B_3 & \beta B_4
 \end{array} \right]$ is positive definite and the determinant is
\begin{eqnarray}
 \nonumber   \left.
\begin{array}{ll}
AB_4 - \gamma^2 \beta B_3^2= \frac{4}{3} h_1^5 h_2^4 h_3 \alpha_1^3 \alpha _2^2+\frac{5}{3} h_1^4 h_2^5 h_3 \alpha_1^2 \alpha_2^3 +\frac{1}{3} h_1^3 h_2^6 h_3 \alpha_1 \alpha_2^4 & \\
 +4 h_1^5 h_2^3 h_3^2 \alpha _1^3 \alpha _2 \alpha _3 +\frac{16}{3} h_1^4 h_2^4 h_3^2 \alpha_1^2 \alpha_2^2 \alpha_3  +\frac{4}{3} h_1^2 h_2^4 h_3^4 \alpha_1^2 \alpha_2^2 \alpha_3& \\
 +\frac{4}{3} h_1^3 h_2^5 h_3^2 \alpha_1 \alpha_2^3 \alpha _3+\frac{2}{3} h_1 h_2^5 h_3^4 \alpha_1 \alpha_2^3 \alpha_3 +h_1 h_2^5 h_3^2 \alpha_1 \alpha_2^3 \tilde{\alpha}_3 \left(\frac{4}{3}h_1^2-h_1h_3 +2 h_3^2\right)  &\\
 +\frac{1}{12} h_2^6 h_3^4 \alpha_2^4 \alpha_3 +\frac{1}{4}h_2^6 h_3^4 \alpha_2^4 \tilde{\alpha_3}+4 h_1^2 h_2^3 h_3^5 \alpha_1^2 \alpha_2 \alpha_3^2 +\frac{7}{3} h_1 h_2^4 h_3^5 \alpha_1 \alpha_2^2 \alpha _3^2&\\
+\frac{1}{3} h_2^5 h_3^5 \alpha_2^3 \alpha_3^2+4 h_1^5 h_2^3 h_3^2 \alpha_1^3 \alpha_2 \tilde{\alpha}_3 +\frac{19}{3} h_1^4 h_2^4 h_3^2 \alpha_1^2 \alpha_2^2 \tilde{\alpha}_3  +12 h_1^5 h_2^2 h_3^3 \alpha_1^3 \alpha_3 \tilde{\alpha}_3 & \\
 +\frac{4}{3} h_2^5 h_3^5 \alpha_2^3 \alpha_3 \tilde{\alpha}_3+\frac{28}{3} h_1 h_2^4 h_3^5 \alpha_1 \alpha_2^2 \alpha_3 \tilde{\alpha}_3&\\
  + h_1^2 h_2^3 h_3^3 \alpha_1^2 \alpha_2 \alpha_3 \tilde{\alpha}_3 \left(20 h_1^2-12h_1h_3 +16 h_3^2\right)  +\frac{4}{3} h_1^3 h_2^4 h_3^3 \alpha_1 \alpha_2^2 \alpha_3 \tilde{\alpha}_3 &\\
+12 h_1^2 h_2^2 h_3^6 \alpha_1^2 \alpha_3^2 \tilde{\alpha}_3 +8 h_1 h_2^3 h_3^6 \alpha_1 \alpha_2 \alpha_3^2 \tilde{\alpha}_3+\frac{4}{3} h_2^4 h_3^6 \alpha_2^2 \alpha_3^2 \tilde{\alpha}_3>0.
 \end{array} \right.
\end{eqnarray}

Define  the operator $\mc A: {\text{Dom}}(\mc A)\subset \mc H \to \mc H$ with ${\rm {Dom}}(\mc A) =  \{(\vec z, \vec {\tilde z}) \in V\times V, A\vec z \in \mathrm V' \},$
where $\mc A= \left[ {\begin{array}{*{20}c}
   0_{2\times 2}  & I_{2\times 2} \\
       A  &  0_{2\times 2}  \\
\end{array}} \right]$ and
\begin{eqnarray}
\label{hom-A} A= \left[ {\begin{array}{*{20}c}
   -\frac{A}{m}D_x^4 -\frac{\gamma\beta h_2 h_3 \varsigma B_1^2}{m} D_x J D_x  & -\frac{\gamma\beta h_2 h_3 \varsigma B_1 B_2}{m} D_x J -\frac{\gamma\beta B_3}{m} D_x^3 \\
     \frac{\gamma\beta h_2 h_3 \varsigma B_1 B_2}{\mu} J D_x + \frac{\gamma \beta B_3}{\mu} D_x^3 & \frac{\beta B_4}{\mu h_3}D_x^2 + \frac{\gamma\beta h_2 \varsigma B_2^2}{\mu} J
\end{array}}  \right],
\end{eqnarray}
\begin{eqnarray}
 \nonumber   \left.
\begin{array}{ll}
{\rm {Dom}}(A) = \left(H^4(0,L)\cap H^2_L(0,L)\right) \times (H^3(0,L)\cap H^1_L(0,L))    \bigcap  & \\
\{ \vec z=(z_1, z_2)^{\rm T}, A\vec z \in \mathrm H, \left.~(z_1)_{xx}=(z_2)_x\right|_{x=L} = 0, & \\
\left.  \quad A {z_1}_{xxx} +\gamma \beta h_2 h_3 B_1^2 \varsigma (J {z_1}_{x}) +\gamma\beta h_2 h_3 B_1 B_2 \varsigma (J z_2) +\gamma\beta B_3{z_2}_{xx}\in H^1_R(0,L)\right\}
 \end{array} \right.
\end{eqnarray}
where $H^1_R(0,L)=\{\varphi\in H^1(0,L): ~\varphi(L)=0\}.$ It is now obvious that $ \mc A z \in \mc H$ as $z\in {\rm {Dom}}(\mc A).$ Define the control operator $B$
 \begin{eqnarray}
\nonumber   && B_0 \in \mathcal{L}(\mathbb{C}, \mathrm V'), ~ \text{with} ~ B_0 =   \left[ \begin{array}{c}
0 \\
-\frac{1}{\mu h_3}  \delta_L
 \end{array} \right], \\
 \nonumber && B \in \mathcal{L}(\mathbb{C} , {\rm Dom}(\mc A)'), ~ \text{with} ~ B=   \left[ \begin{array}{c} 0_{2\times 1} \\ B_0 \end{array} \right].
 \end{eqnarray}
where  $\mathrm V'$ is the dual of $\mathrm V$ pivoted with respect $\mathrm H.$

Writing $\varphi=[w, p, \dot w, \dot p]^{\rm T},$ the control  system (\ref{perturbeddd-an})-(\ref{ivpp})  with the voltage controller $V(t)$ can be put into the  state-space form
\begin{eqnarray}
\label{Semigroup-mm} \left\{\begin{array}{ll}
\dot \varphi =  \underbrace{\left[ {\begin{array}{*{20}c}
   0 & I_{2\times 2}  \\
   -A & 0  \\
\end{array}} \right]}_{\mc A}  \varphi +\underbrace{ \left[ \begin{array}{c} 0 \\ B_0 \end{array} \right]}_B V(t) , \quad \varphi(x,0) =  \varphi ^0.
\end{array}\right.
\end{eqnarray}
 \begin{lemma} \label{skeww}The operator $\mc{A}$  satisfies $\mc{A}^*=-\mc{A}$ on  $\mc{H},$ and
 \begin{eqnarray}\nonumber {\rm Re}\left<\mc{A} \varphi, \varphi\right>_{\mc{H}}={\rm Re}\left<\mc{A}^*\varphi, \varphi\right>_{\mc{H}}= 0.\end{eqnarray}
 Also, $\mc{A} $ has a compact resolvent.
\end{lemma}
\begin{proof} Let $\vec y=\left[ \begin{array}{l}
 y_1\\
 y_2
 \end{array} \right], \vec z=\left[ \begin{array}{l}
 z_1\\
 z_2
 \end{array} \right] \in {\rm Dom}(\mc A).$ Integration by parts and using the boundary conditions (\ref{ivpp}) yield

\begin{eqnarray}
 \nonumber   \left.
\begin{array}{ll}
   \left<\mc A \vec y, \vec z\right>_{\mc H} = \int_0^L\left\{  -A(y_1)_{xxxx}\bar z_3     -\gamma\beta h_2 h_3 \varsigma B_1^2 (J (y_1)_x)_x \bar z_3  \right.&\\
\quad -\gamma\beta h_2 h_3 \varsigma B_1 B_2  (J y_2)_x \bar z_3 -\gamma\beta B_3 (y_2)_{xxx}\bar z_3 + \gamma\beta h_2 h_3 \varsigma B_1 B_2 J (y_1)_x \bar z_4 &\\
\quad + \gamma \beta h_3 B_3 (y_1)_{xxx} \bar z_4+  \beta B_4(y_2)_{xx} \bar z_4 + \gamma\beta h_2h_3 \varsigma B_2^2 J y_2 \bar z_4&\\
 \quad-\gamma\beta h_2 h_3 \varsigma (B_1J (y_3)_x + B_2 Jy_4) (B_1 (\bar z_1)_x + B_2 \bar z_2)+A(y_3)_{xx}(\bar z_1)_{xx} & \\
 \left.\quad+ \gamma\beta B_3  (y_4)_x (\bar z_1)_{xx}+ \gamma \beta B_3 (y_3)_{xx} (\bar z_2)_x+  \beta B_4(y_4)_{x} (\bar z_2)_x \right\} dx\\
     = \int_0^L\left\{  -A(y_1)_{xx}(\bar z_3)_{xx} +\gamma\beta h_2 h_3 \varsigma B_1^2 J y_1 (\bar z_3)_x    +\gamma\beta h_2 h_3 \varsigma B_1 B_2  J y_2) (\bar z_3)_x \right.&\\
 \quad +\gamma\beta B_3 (y_2)_{xx}(\bar z_3)_x  - \gamma\beta h_2 h_3 \varsigma B_1 B_2  y_1 (J\bar z_4)_x- \gamma \beta h_3 B_3 (y_1)_{xx} (\bar z_4)_x &\\
    \quad-  \beta B_4(y_2)_{x} (\bar z_4)_x+ \gamma\beta h_2h_3 \varsigma B_2^2 y_2 J\bar z_4 -\gamma\beta h_2 h_3 \varsigma \left(-B_1^2 (J(y_3)_x)\bar z_1 \right.&\\
     \left.\quad  - B_1 B_2 (J (\bar z_2)_x) y_3-B_1 B_2 (Jy_4)_x \bar z_1  + B_2 ^2 Jy_4 \bar z_2 \right) +A(y_3)_{xxxx}\bar z_1 &\\
     \left. \quad- \gamma\beta B_3  (y_4)_{xx} (\bar z_1)_{x} -\gamma \beta B_3 (y_3)_{x} (\bar z_2)_{xx}-  \beta B_4 y_4 (\bar z_2)_{xx} \right\} ~dx&\\
    =\left< \vec y, \mc A^* \vec z\right>_{\mc H} &\\
    =\left< \vec y, -\mc A \vec z\right>_{\mc H}.
 \end{array} \right.
\end{eqnarray}

This shows that $\mc A$ is skew-symmetric. To prove that $\mc A$ is skew-adjoint on $\mc H,$ i.e. $\mc A^*=-\mc A$ on $\mc H,$   it is required to show that  for any  $v\in \mathrm{H}$  there is $u \in  \text{Dom}(\mc A )$ so that   $\mc A u=v,$ see \cite[Proposition 3.7.3]{Weiss-Tucsnak}.
This is equivalent to solving the  system of equations  for $ u \in  \text{Dom}(\mc A ).$
Using (\ref{coef}) to simplify the equations leads to
\begin{eqnarray}
\nonumber  m v_3 &=&-A(u_1)_{xxxx} -\gamma\beta h_2 h_3 \varsigma B_1^2 (J (u_1)_x)_x \\
 \label{cra1} && ~~  -\gamma\beta h_2 h_3 \varsigma B_1 B_2  (J u_2)_x -\gamma\beta B_3 (u_2)_{xxx} \\
\nonumber  \mu h_3 v_4 &=& \gamma\beta h_2 h_3 \varsigma B_1 B_2 J (u_1)_x + \gamma \beta h_3 B_3 (u_1)_{xxx}\\
 \label{cra2} && ~~+  \beta B_4(u_2)_{xx}  + \gamma\beta h_2h_3 \label{cra3} \varsigma B_2^2 (J u_2) \\
\nonumber  v_1&=&u_3\\
\nonumber  v_2&=&u_4.
\end{eqnarray}
Define the bilinear forms $a$ and $c$ by
\begin{eqnarray}
\nonumber a\left(\left[ \begin{array}{l}
 u_1\\
 u_2
 \end{array} \right], \left[ \begin{array}{l}
 F\\
 G
 \end{array} \right]\right)&=& -\gamma \beta h_2 h_3 \varsigma\left< J(B_1 (u_1)_x + B_2 u_2),B_1 (v_1))_x + B_2  v_2  \right>_{\mX} \\
 \nonumber && + \left<\left[ \begin{array}{ll}
 A & \gamma\beta B_3\\
 \gamma\beta B_3 & \beta B_4
 \end{array} \right]\left[ \begin{array}{l}
 (u_1)_{xx}\\
 (u_2)_x
 \end{array} \right], \left[ \begin{array}{l}
 (v_1)_{xx}\\
  (v_2)_{x}
 \end{array} \right]\right>_{\mX^2},\\
 \nonumber c\left(\left[ \begin{array}{l}
 u_1\\
 u_2
 \end{array} \right], \left[ \begin{array}{l}
 F\\
 G
 \end{array} \right]\right)&=&m\left<u_1,v_1\right>_{\mX}+\mu h_3\left<u_2, v_2\right>_{\mX}.
 \end{eqnarray}
 Let $(F,G)\in H^2_L(0,L) \times H^1(0,L).$ If we multiply (\ref{cra1}) by $F$ and (\ref{cra2}) by  $G,$ then integrate by parts, we obtain
\begin{eqnarray}a\left(\left[ \begin{array}{l}
 u_1\\
 u_2
 \end{array} \right], \left[ \begin{array}{l}
 F\\
 G
 \end{array} \right]\right)+c\left(\left[ \begin{array}{l}
 u_1\\
 u_2
 \end{array} \right], \left[ \begin{array}{l}
 F\\
 G
 \end{array} \right]\right)=0.\label{form}\end{eqnarray}
The bilinear forms $a$ and $c$ are symmetric, bounded and coercive $H^2_L(0,L)$ and $H^1_L(0,L),$ respectively. Therefore, by Lax-
Milgram theorem, there exists a unique pair $(u_1, u_2)\in H^2_L(0,L) \times H^1_L(0,L)$
satisfying (\ref{form}).

The last step of our proof is to show that $\left[ \begin{array}{l}
 u_1\\
 u_2
 \end{array} \right]\in {\rm Dom}(A).$ However, this step is straight forward since (\ref{form}) is the definition of ${\rm Dom}(A).$ To see this let $\left[ \begin{array}{l}
 u_1\\
 u_2
 \end{array} \right] \in H^2_L(0,L)\times H^1_L(0,L),$ $\left[ \begin{array}{l}
F\\
G
 \end{array} \right]\in H^2_L(0,L)\times H^1_L(0,L),$ and all boundary conditions  hold. A simple calculation shows that
\begin{eqnarray}
\nonumber &&c\left(A \left[ \begin{array}{l}
 u_1\\
 u_2
 \end{array} \right], \left[ \begin{array}{l}
 F\\
 G
 \end{array} \right] \right) \\
 \nonumber &&=c\left( \left[ \begin{array}{l}
 \frac{1}{m}\left(-A(u_1)_{xxxx} -\gamma\beta h_2 h_3 \varsigma \left(B_1^2 (J (u_1)_x)_x   \right.\right.\\
 \frac{1}{\mu h_3}\left(\gamma\beta h_2 h_3 \varsigma \left(B_1 B_2 J (u_1)_x +B_2^2 (J u_2)\right) \right)
 \end{array} \right. \right.\\
 \nonumber &&\quad\quad\left. \left. \begin{array}{l}
   \left.\left. + B_1 B_2  (J u_2)_x\right) -\gamma\beta B_3 (u_2)_{xxx}\right)\\
        \left.\gamma \beta h_3 B_3 (u_1)_{xxx} +  \beta B_4(u_2)_{xx} \right)
 \end{array} \right], \left[ \begin{array}{l}
 F\\
 G
 \end{array} \right] \right)\\
 \nonumber &&=\left< \left[ \begin{array}{l}
-A(u_1)_{xxxx} -\gamma\beta h_2 h_3 \varsigma \left(B_1^2 (J (u_1)_x)_x  \right. \\
\gamma\beta h_2 h_3 \varsigma \left(B_1 B_2 J (u_1)_x +B_2^2 (J u_2)\right)
 \end{array} \right.\right. \\
\nonumber &&\quad\quad\left. \left. \begin{array}{l}
\left.  + B_1 B_2  (J u_2)_x\right) -\gamma\beta B_3 (u_2)_{xxx}\\
 \gamma \beta h_3 B_3 (u_1)_{xxx} +  \beta B_4(u_2)_{xx}
 \end{array} \right], \left[ \begin{array}{l}
 F\\
 G
 \end{array} \right] \right>_{\mX^2}\\
  \nonumber && = -a\left(\left[ \begin{array}{l}
 u_1\\
 u_2
 \end{array} \right], \left[ \begin{array}{l}
 F\\
 G
 \end{array} \right] \right),\end{eqnarray}
and therefore
\begin{eqnarray}\label{salak2} a\left(\left[ \begin{array}{l}
 u_1\\
 u_2
 \end{array} \right], \left[ \begin{array}{l}
 F\\
 G
 \end{array} \right]\right)+c\left(\left[ \begin{array}{l}
 u_1\\
 u_2
 \end{array} \right], \left[ \begin{array}{l}
 F\\
 G
 \end{array} \right] \right)=0.
 \end{eqnarray}
 Now if we let $\left[ \begin{array}{l}
 F\\
G
 \end{array} \right]=\left[ \begin{array}{l}
 F_1\\
 G_1
 \end{array} \right]\in (C_0^{\infty}(0,L))^2,$ it follows that
 $$a\left(\left[ \begin{array}{l}
 u_1\\
 u_2
 \end{array} \right], \left[ \begin{array}{l}
 F_1\\
 G_1
 \end{array} \right] \right)+c\left(\left[ \begin{array}{l}
 u_1\\
 u_2
 \end{array} \right], \left[ \begin{array}{l}
 F_1\\
 G_1
 \end{array} \right] \right)=0$$
 holds for all $\left[ \begin{array}{l}
 F_1\\
 G_1
 \end{array} \right]\in (C_0^{\infty}(0,L))^2.$ Therefore in $((C_0^{\infty}(0,L))^2)'$ we have
 \begin{eqnarray}
\label{salak} \left\{
  \begin{array}{ll}
A(u_1)_{xxxx} +\gamma\beta h_2 h_3 \varsigma \left[B_1^2 (J (u_1)_x)_x + B_1 B_2  (J u_2)_x\right] +\gamma\beta B_3 (u_2)_{xxx}=0, & \\
  \gamma\beta h_2 h_3 \varsigma \left[B_1 B_2 J (u_1)_x +B_2^2 (J u_2)\right] + \gamma \beta h_3 B_3 (u_1)_{xxx} +  \beta B_4(u_2)_{xx}
=0. &
 \end{array} \right.
\end{eqnarray}
 If we substitute (\ref{salak}) into (\ref{form}) we obtain
that $\left[ \begin{array}{l}
 u_1\\
 u_2
 \end{array} \right]$ satisfies (\ref{salak2}). This together with (\ref{salak}) implies that $\left[ \begin{array}{l}
 u_1\\
 u_2
 \end{array} \right]\in {\rm Dom}(A).$
 \end{proof}

Consider all external volume and surface forces are zero. The only external force acting on the system is purely electrical, i.e. $V(t)\ne 0.$ We have the following well-posedness theorem for (\ref{Semigroup-mm}).
\begin{theorem}\label{w-pff}
Let $T>0,$ and $V(t)\in \Ltwo(0,T).$ For any $\varphi^0 \in \mc{H},$ $\varphi\in C[[0,T]; \mc H]$ and there exists a positive constants $c_1(T)$
such that (\ref{Semigroup-mm}) satisfies
      \begin{eqnarray}\nonumber\|\varphi (T) \|^2_{\mc{H}} &\le& c_1 (T)\left\{\|\varphi^0\|^2_{\mc{H}} + \|V\|^2_{\Ltwo(0,T)}\right\}.
      \end{eqnarray}
\end{theorem}
\begin{proof}   The operator $\mc A: {\rm {Dom}}(\mc A) \to \mc H$ is the  infinitesimal generator of $C_0-$semigroup of contractions by L{\"{u}}mer-Phillips theorem by using Lemma \ref{skeww}.  The operator $B$ defined above is an admissible control operator. The conclusion follows.
\end{proof}
\subsection{Electrostatic Mead-Marcus (M-M) model}

Notice that if the magnetic effects in (\ref{perturbed-dumb}) are neglected, i.e. $\ddot \mu p\equiv 0,$ the last equation can be solved for $p_{xx}.$ Then we obtain the following model
 \begin{eqnarray}
\label{abstractMM}\left\{ \begin{array}{ll}
m \ddot w + {\tilde A} w_{xxxx} -  \beta \gamma h_2h_3 \varsigma{\tilde B}  \phi^2_x=-\frac{\gamma B_3}{B_4}V(t) (\delta_L)_x, &\\
\varsigma{\tilde C} \phi^2 -\phi^2_{xx} + {\tilde B} w_{xxx} =-\frac{B_2}{\beta B_4}V(t) \delta_L,
\end{array}\right.
 \end{eqnarray}
with the simplified boundary conditions
 \begin{eqnarray}
\label{MM-bc}
\begin{array}{ll}
  w(0)=w_x(0)=\phi^2(0)= w_{xx}(L)=\phi^2_x(L)=0,&\\
 - \tilde A w_{xxx}(L) + \beta \gamma h_2h_3 \varsigma{\tilde B} \phi^2(L)=g(t).
 \end{array} \end{eqnarray}
Here the coefficients $\tilde A=A-\frac{\gamma^2 \beta B_3^2}{B_4}>0, \tilde B=B_1-\frac{\gamma B_2 B_3}{B_4}>0, \tilde C=C+ \frac{\gamma h_2 h_3 B_2^2}{B_4} $ are functions of the materials parameters.  This model fits in the form of the abstract Mead-Marcus beam model obtained in \cite{Hansen3,O-Hansen1}. The main difference is the boundary conditions.

Using the definitions of  $J=P_\xi D_x^2$ with $P_\xi=(\varsigma \tilde C I -D_x^2)^{-1}$  in Lemma \ref{pxi},   (\ref{abstractMM})-(\ref{MM-bc}) can be simplified to
\begin{eqnarray}
\nonumber &&m \ddot w + {\tilde A} w_{xxxx}+\gamma\beta \varsigma h_2 h_3 \tilde B^2 (J w_{x})_x=\\
\label{MM} &&\quad-\frac{\gamma}{B_4}\left[ \varsigma h_2 h_3  \tilde B B_2 (P_\xi \delta_L)_x + B_3(\delta_L)_x\right]V(t)\quad\quad
    \end{eqnarray}
    with the boundary conditions
 \begin{eqnarray}
\label{MM-b}\begin{array}{ll}
 w(0)=w_x(0)= w_{xx}(L)=0, ~~
  {\tilde A} w_{xxx}(L)+\gamma\beta \varsigma h_2 h_3 \tilde B^2 J w_{x}(L)=g(t).
 \end{array} \end{eqnarray}

\noindent {\bf{Semigroup well-posedness:}} Consider only the case $g(t)\equiv 0.$ Define $\mc{H} =\mathrm V \times \mathrm H= H^2_L(0,L) \times \Ltwo(0,L).$
The  energy associated with (\ref{perturbeddd-an})-(\ref{ivpp}) is
\begin{eqnarray}
\nonumber  \mathrm{E}(t) =\frac{1}{2}\int_0^L \left\{ m |\dot w|^2  + \tilde A |w_{xx}|^2 -\gamma\beta \varsigma h_2 h_3 \tilde B^2 J w_{x} \bar w_x \right\} ~dx.
\end{eqnarray}
 This motivates definition of the inner product on $\mc{H}$
\begin{eqnarray}
\nonumber && \left<\left[ \begin{array}{l}
 u_1 \\
 u_2
 \end{array} \right], \left[ \begin{array}{l}
 v_1 \\
 v_2
 \end{array} \right]\right>_{\mc{H}}= \left<u_2, v_2\right>_{\mathrm H} + \left<u_1, v_1\right>_{\mathrm V}\\
 \nonumber  &&\quad\quad= \int_0^L \left\{m  u_2 { {\bar v}_2}+ \tilde A (u_2)_{xx} (\bar v_2)_{xx}  -\gamma\beta \varsigma h_2 h_3 \tilde B^2 J (u_1)_{x}  (\bar u_1)_x  \right\}~dx.
 \end{eqnarray}
Define  the operator $\mc A: {\text{Dom}}(\mc A)\subset \mc H \to \mc H$
where
\begin{eqnarray}\label{A-MM}\mc A= \left[ {\begin{array}{*{20}c}
   0 & I \\
       \frac{-1}{m}\left(\tilde A D_x^4+\gamma\beta \varsigma h_2 h_3 \tilde B^2 D_x J D_x\right) &  0  \\
\end{array}} \right]
\end{eqnarray} with
\begin{eqnarray}
 \nonumber   \left.
\begin{array}{ll}
{\rm {Dom}}(\mc A) = \{ (z_1,z_2)\in \mc H, z_2\in H^2_L(0,L),  (z_1)_{xx}(L)=0,&\\
\tilde A (z_1)_{xxx}+\gamma\beta \varsigma h_2 h_3 \tilde B^2  J (z_1)_x \in H^1(0,L),&\\
\tilde A (z_1)_{xxx}(L)+\gamma\beta \varsigma h_2 h_3 \tilde B^2  J (z_1)(L)=0 \}. &
 \end{array} \right.
\end{eqnarray}
 Define the control operator $B \in \mathcal{L}(\mathbb{C} , {\rm Dom}(\mc A)')$ by
 \begin{eqnarray}
\nonumber B=   \left[ \begin{array}{c} 0 \\ -\frac{\gamma}{mB_4}\left[ \varsigma h_2 h_3  \tilde B B_2(P_\xi \delta_L)_x + B_3(\delta_L)_x\right] \end{array} \right]
 \end{eqnarray}
and the dual operator $B^*\in \mathcal{L}( \mc H, \mathbb C)$ is  $$B^*\Phi=\frac{\gamma}{m B_4}\left[ \varsigma h_2 h_3  \tilde B B_2(P_\xi (\Phi_2)_x(L))  + B_3(\Phi_2)_x(L)\right].$$
Writing $\varphi=[w,\dot w]^{\rm T},$ the control  system (\ref{MM})-(\ref{MM-b})  with the voltage controller $V(t)$ can be put into the  state-space form
\begin{eqnarray}
\label{Semigroupp-mmm}\left\{
\begin{array}{ll}
\dot \varphi = {\mc A}  \varphi +B V(t) , \quad \varphi(x,0) =  \varphi ^0.
\end{array}\right.
\end{eqnarray}
\begin{theorem}\label{w-pff}
Let $T>0,$ and $V(t)\in \Ltwo(0,T).$ For any $\varphi^0 \in \mc{H},$ $\varphi\in C[[0,T]; \mc H]$ and there exists a positive constants $c_1(T)$
such that (\ref{Semigroupp-mmm}) satisfies
      \begin{eqnarray}\label{conc}\|\varphi (T) \|^2_{\mc{H}} &\le& c_1 (T)\left\{\|\varphi^0\|^2_{\mc{H}} + \|V\|^2_{\Ltwo(0,T)}\right\}.
      \end{eqnarray}
\end{theorem}
\begin{proof}
    The proof follows the steps of Theorem \ref{w-pff}.
\end{proof}


\section{Stabilization results}
\label{stab}
 In this paper, the top and bottom layers are made of different materials, one is elastic and another one is piezoelectric, in contrast to model in \cite{Ozkan3}. This causes different speeds of wave propagation at the top and bottom layers.  In fact, if both the top and the bottom layers of the composite are piezoelectric, it is shown in \cite{Ozkan3} that the fully dynamic R-N model (\ref{Semigroup-non})   with four $B^*-$type  feedback controllers lacks of asymptotic stability for many choices of material parameters of the piezoelectric layers. These solutions are corresponding to the ``bending-free" or ``inertial sliding" motions.  In contrast to this result, inertial-sliding solutions are asymptotically stable. We also show that the electrostatic R-N model is shown to be exponentially stable with only three   feedback controllers  in comparison to the four feedback controllers in \cite{Ozkan3}. We improve that result.

The fully dynamic M-M model  has unstable solutions with the $B^*-$feedback. This is in line with the result in \cite{Ozkan3} that the  material parameters are sensitive to the stabilization of the system with only one feedback controller. The electrostatic M-M model is asymptotically  stable with only one controller.  We also mention an exponential stability result for the electrostatic M-M beam model at the end without a proof since the proof is too long and it is beyond the scope of this paper.
\subsection{The fully dynamic Rao-Nakra (R-N) model}

Let   $k_1,k_2,k_3,k_4\in\mathbb{R}^+,$  and choose the state feedback controllers in the fully-dynamic model (\ref{Semigroup}) as the following
\begin{eqnarray}\label{feedback-f}
{\bf F}(t)=\left(
                                                                                                               \begin{array}{c}
                                                                                                                 g^1(t) \\
                                                                                                                 V(t) \\
                                                                                                                 M(t) \\
                                                                                                                 g(t) \\
                                                                                                               \end{array}
                                                                                                             \right)=KB^*\varphi=\left(
                                                                                                               \begin{array}{c}
                                                                                                                 -k_1\dot v^1(L) \\
                                                                                                                 k_2 \dot p(L) \\
                                                                                                                 k_3\dot w_x(L) \\
                                                                                                                 -k_4 \dot w(L) \\
                                                                                                               \end{array}
                                                                                                             \right)
\end{eqnarray}
where  $K={\rm diag}(-k_1,k_2,k_3,-k_4),$ and $\dot v^1(L), \dot p(L), \dot w(L), \dot w_x (L)$ are the velocity of the elastic layer, total induced current accumulated at the electrodes of the  piezoelectric layer, angular velocity and  velocity of the bending of the composite at the tip $x=L,$ respectively.

Consider the eigenvalue problem $\mc A \varphi=\lambda\varphi$ for the inertial sliding solutions, i.e. $w=0:$
\begin{eqnarray}
 \label{inertial-sliding} &&\left\{
  \begin{array}{ll}
  \alpha^1 h_1 v^1_{xx}- \frac{G_2}{h_2} (-v^1+v^3) = -\tau^2 \rho_1 h_1 v^1,    & \\
 \alpha^3_1 h_3 v^3_{xx}-\gamma \beta h_3 p^i_{xx} +\frac{G_2}{h_2} (-v^1+v^3)= -\tau^2 \rho_3 h_3 v^3,    & \\
    \beta h_3  p_{xx} - \gamma\beta h_3 v^3_{xx}= -\tau^2 \mu  h_3  p,  &\\
   \frac{G_2H}{h_2}   (-v^1+v^3)_x=0,&
 \end{array} \right.
\end{eqnarray}
with the overdetermined boundary conditions
\begin{eqnarray}
 \nonumber   \left|w=w_x=v^i=p^i\right|_{x=0}= \left|v^i_x=p^i_x =p^i\right|_{x=L}=0,\\
  \label{d-son-eig-mag} w(L)=w_x(L)=w_{xx}(L)=w_{xxx}(L)=0, \quad i=1,3.
\end{eqnarray}
By using the boundary conditions $v^1(0)=v^3(0)=0,$ the last equation in (\ref{inertial-sliding}) implies that $v^1=v^3.$ Then (\ref{inertial-sliding}) reduces to
\begin{eqnarray}
 \label{inertial-sliding1} &&\left\{
  \begin{array}{ll}
  \alpha^1 h_1 v^1_{xx}= -\tau^2 \rho_1 h_1 v^1,    & \\
 \alpha^3_1 h_3 v^3_{xx}-\gamma \beta h_3 p^i_{xx} = -\tau^2 \rho_3 h_3 v^3,    & \\
    \beta h_3  p_{xx} - \gamma\beta h_3 v^3_{xx}= -\tau^2 \mu  h_3  p,  &
 \end{array} \right.
\end{eqnarray}
By the boundary conditions for $v^1,$ we obtain $v^1=v^3\equiv 0.$ Finally, by the boundary conditions for $p,$ the equation  $ \beta h_3  p_{xx} = -\tau^2 \mu  h_3  p$ has only the solution $p\equiv 0.$ We have the following immediate result:

\begin{theorem}\label{lack1}
The  inertial sliding solutions of the fully dynamic model (\ref{Semigroup}) is  strongly stable by the  feedback (\ref{feedback-f}).
\end{theorem}

\begin{proof} It can be shown in (\ref{inertial-sliding}) that  $0\in \sigma(\mc A)$ since $\lambda=0$ corresponds to the trivial solution. Therefore, there are only isolated eigenvalues. There are also no eigenvalues on the imaginary axis, or in other words, the set
\begin{eqnarray} \label{set2}\left\{z\in \mc H: {\rm Re} \left<\mc A z, z\right>_{\mc H}= -k_1 |\dot v(L)|^2-k_2 |\dot p(L)|^2=0\right\}
\end{eqnarray}
has only the trivial solution by the argument (\ref{inertial-sliding}-(\ref{inertial-sliding1}). Therefore, by La Salle's invariance principle,
the system is asymptotically stable.
\end{proof}

It is important to note that the asymptotic stability of other solutions is still an open problem due to the strong coupling between bending, stretching, and charge equations.

\subsection{Electrostatic Rao-Nakra (R-N) model}
 Let   $k_1,k_2,k_3,k_4\in\mathbb{R}^+,$  and choose the state feedback controllers in the model with no magnetic effects (\ref{d4-non}) as the following
\begin{eqnarray}\label{feedback}
{\bf F}(t)=\left(
                                                                                                               \begin{array}{c}
                                                                                                                 g^1(t) \\
                                                                                                                 V(t) \\
                                                                                                                 M(t) \\
                                                                                                                 g(t) \\
                                                                                                               \end{array}
                                                                                                             \right)=KB^*\varphi=\left(
                                                                                                               \begin{array}{c}
                                                                                                                 -k_1\dot v^1(L) \\
                                                                                                                 -k_2 \dot v^3(L) \\
                                                                                                                 k_3\dot w_x(L) \\
                                                                                                                 -k_4 \dot w(L) \\
                                                                                                               \end{array}
                                                                                                             \right)
\end{eqnarray}
where  $K={\rm diag}(-k_1,-k_2,k_3,-k_4),$ and $\dot v^3(L)$ is the velocity of the  piezoelectric layer, at the tip $x=L.$

The model without magnetic effects is exponentially stable with the feedback (\ref{feedback}).  We recall the following theorem:
\begin{theorem} \cite[Theorem 4.3]{Ozkan3}
   Let the feedback (\ref{feedback}) be chosen. The solutions $\varphi(t)=e^{(\mc A+ KBB^*)t}\varphi_0$  for $t\in \mathbb{R}^+$ of the closed-loop system
   \begin{eqnarray}
\label{Semi-feed}\left\{\begin{array}{ll}
\dot \varphi =  \mc A \varphi +KB B^*\varphi , \quad \varphi(x,0) =  \varphi ^0.\end{array}\right.
\end{eqnarray}
     is exponentially stable in $\mc H.$
\end{theorem}

 In this paper, we use modified multipliers  to reduce the number of controllers to three, i.e. the controller $g(t)= -k_4 \dot w(L)$ may be removed. To achieve this, the following result plays a key role in order to show that there are no eigenvalues on the imaginary axis. Note that the analogous result in \cite{Ozkan3} requires $u(L)=0.$

\begin{lemma} \label{xyz} Let $\lambda =i\mu.$ The eigenvalue problem
\begin{eqnarray}
 \label{dbas-eig} &&\left\{
  \begin{array}{ll}
 \alpha^1_1 h_1 z^1_{xx} -  G_2 \phi^2 = \lambda^2 \rho_1 h_1 z^1,   & \\
  \alpha^3_1 h_3 z^3_{xx} + G_2 \phi^2 = \lambda^2 \rho_3 h_3 z^3,  & \\
    - K_2 u_{xxxx} + G_2 H \phi^2_x=\lambda^2 (m u  - K_1 u_{xx}),&\\
 \phi^2=\frac{1}{h_2}\left(-z^1+z^3 + H u_x\right)&
 \end{array} \right.
\end{eqnarray}
with the overdetermined boundary conditions
\begin{eqnarray}
\label{d-son-eig} \left.\begin{array}{ll}
  u(0)=u_x(0)=u_x(L)=u_{xx}(L)=u_{xxx}(L)=0,&\\
  z^i(0)=z^i(L)=z^i_x (L)=0, \quad i=1,3,
\end{array}\right.&&
   \end{eqnarray}
has only the trivial solution.
\end{lemma}
\begin{proof} Now multiply the equations in (\ref{dbas-eig}) by  $x \bar z^1_{xxx},$  $ x\bar z^3_{xxx},$ and $x\bar u_{xxx},$ respectively, use the boundary conditions (\ref{d-son-eig}), integrate by parts on $(0,L),$ and add them up:
\begin{eqnarray}
 \label{CH3-mult1-20}  \left.
\begin{array}{ll}
 0=\int_0^L\left\{ -\alpha^1_1 h_1|z^1_{xx}|^2 -\alpha^3_1 h_3 |z^3_{xx}|^2   -K_2 |u_{xxx}|^2 -3\rho_1 h_1 \lambda^2|z^1_{x}|^2  \right.& \\
  -3\rho_3 h_3  \lambda^2 |z^3|^2 +3m \lambda^2 |u_x|^2 + K_1 \lambda^2|u_{xx}|^2 + 3G_2  H |\phi^2_x|^2 &\\
-G_2H {\bar \phi^2} x \phi^2_{xxx} -K_2 \bar u_{xxxx} (x u_{xxx})  - G_2 H \phi^2 {\bar u}_{xxx}-\alpha^3_1 h_3 \bar z^3_{xx}(x z^3_{xxx})  &\\
 - \alpha^1_1 h_1 \bar z^1_{xx} (x z^1_{xxx})-\rho_1h_1 \lambda^2\bar z^1 (xz^1_{xxx})  -\rho_1h_1 \lambda^2\bar z^1 (xz^1_{xxx})    &\\
\left.   -\rho_3h_3 \lambda^2 \bar z^3 (x z^3_{xxx})  +\lambda^2 (m\bar u-K_1 \bar u_{xx})(x u_{xxx})\right\}~dx&\\
 \end{array} \right.
\end{eqnarray}
 where we use the boundary conditions for $z^1_{xx}(L)=z^3_{xx}(L)=0$ via the differential equations (\ref{dbas-eig}) since $\phi^2(L)=z^1(L)=z^3(L)=0.$

Now consider the conjugate eigenvalue problem corresponding to (\ref{dbas-eig})-(\ref{d-son-eig}):
\begin{eqnarray}
 \label{dbas-eig-con} &&\left\{
  \begin{array}{ll}
 \alpha^1_1 h_1 {\bar z}^1_{xx} -  G_2 {\bar \phi}^2 = {\bar \lambda}^2 \rho_1 h_1 {\bar z}^1,   & \\
  \alpha^3_1 h_3 {\bar z}^3_{xx} + G_2 {\bar \phi}^2 = {\bar \lambda}^2 \rho_3 h_3 {\bar z}^3,  & \\
    - K_2 {\bar u}_{xxxx} + G_2 H {\bar \phi}^2_x={\bar \lambda}^2 (m {\bar u}  - K_1 {\bar u}_{xx}),&\\
 {\bar \phi}^2=\frac{1}{h_2}\left(-{\bar z}^1+{\bar z}^3 + H {\bar u}_x\right)&
 \end{array} \right.
\end{eqnarray}
with overdetermined boundary conditions
\begin{eqnarray}
 \nonumber   &&  \bar u(0)=\bar u_x(0)=\bar u_x(L)=\bar u_{xx}(L)=\bar u_{xxx}(L)=0,\\
 \label{d-son-eig-con} && \bar z^i(0)=\bar z^i(L)=\bar z^i_x (L)=0, \quad i=1,3.
\end{eqnarray}
Now multiply the equations in (\ref{dbas-eig-con}) by  $x z^1_{xxx},$  $ x z^3_{xxx},$ and $x u_{xxx},$ respectively, integrate by parts on $(0,L),$ and add them up:
\begin{eqnarray}
\label{CH3-mult1-21}  \left.
\begin{array}{ll}
 0=\int_0^L\left\{ G_2 h_2 \bar \phi^2 (x\phi^2_{xxx})+G_2 h_2 H \phi^2 {\bar u}_{xxx} \right.   +K_2 \bar u_{xxxx} (x u_{xxx})  & \\
 +\alpha^1_1 h_1 \bar z^1_{xx} (x z^1_{xxx}) +\alpha^3_1 h_3 \bar z^3_{xx}(x z^3_{xxx}) +\rho_1h_1 {\bar \lambda^2}\bar z^1 (xz^1_{xxx})  &\\
\left.  +\rho_3h_3 {\bar \lambda}^2 \bar z^3 (x z^3_{xxx})  -{\bar \lambda}^2 (m\bar u-K_1 \bar u_{xx})(x u_{xxx})\right\}~dx.
 \end{array} \right.&&
\end{eqnarray}
 Since $\lambda = i\mu,$ adding (\ref{CH3-mult1-20}) and (\ref{CH3-mult1-21}) yields
 \begin{eqnarray}
\label{CH3-mult1-20-a}  \left.
\begin{array}{ll}
 0=\int_0^L\left\{ -\alpha^1_1 h_1|z^1_{xx}|^2 -\alpha^3_1 h_3 |z^3_{xx}|^2  -K_2 |u_{xxx}|^2 -3\rho_1 h_1 \lambda^2|z^1_{x}|^2  \right.& \\
 \left.-3\rho_3 h_3  \lambda^2 |z^3|^2 +3m \lambda^2 |u_x|^2+ K_1 \lambda^2|u_{xx}|^2  + 3G_2  H |\phi^2_x|^2 \right\} ~dx.
 \end{array} \right.&&
\end{eqnarray}
 Now multiply the equations in (\ref{dbas-eig}) by  $3 \bar z^1_{xx},$  $ 3\bar z^3_{xx},$ and $3\bar u_{xx},$ respectively, integrate by parts on $(0,L),$ and add them up:
   \begin{eqnarray}
\label{CH3-mult1-22} \left.
\begin{array}{ll}
0=\int_0^L\left\{ 3\alpha^1_1 h_1|z^1_{xx}|^2 +3\alpha^3_1 h_3 |z^3_{xx}|^2 +3K_2 |u_{xxx}|^2-3G_2 H |\phi^2_x|^2  \right.  & \\
\left. +3\rho_1 h_1 \lambda^2|z^1_{x}|^2 +3\rho_3 h_3  \lambda^2 |z^3|^2-3m \lambda^2 |u_x|^2 -K_1 \lambda^2|u_{xx}|^2   \right\} dx.
 \end{array} \right.&&
\end{eqnarray}
 Finally, adding (\ref{CH3-mult1-20-a}) and (\ref{CH3-mult1-22}) yields
 \begin{eqnarray}
\label{CH3-mult1-200}   \int_0^L\left\{ \alpha^1_1 h_1|z^1_{xx}|^2 +\alpha^3_1 h_3 |z^3_{xx}|^2  +K_2 |u_{xxx}|^2  \right\} ~dx=0.
 \end{eqnarray}
 This implies that $z^1_{xx}=z^3_{xx}=u_{xxx}=0,$ and by using the overdetermined boundary conditions (\ref{d-son-eig}), we obtain that $z^1=z^3=u\equiv 0. $
 \end{proof}

 Let the controller $g(t)$ be removed in (\ref{feedback}). Now the number of feedback controllers is reduced to three:
 \begin{eqnarray}\label{feedback-red}
\left(
                                                                                                               \begin{array}{c}
                                                                                                                 g^1(t) \\
                                                                                                                 V(t) \\
                                                                                                                 M(t) \\
                                                                                                               \end{array}
                                                                                                             \right)=B^*\varphi=\left(
                                                                                                               \begin{array}{c}
                                                                                                                 -k_1\dot v^1(L) \\
                                                                                                                 k_2 \dot v^3(L) \\
                                                                                                                 -k_3\dot w_x(L) \\
                                                                                                               \end{array}
                                                                                                             \right).
\end{eqnarray}
 \begin{theorem}\label{RN-elec}
   Let the feedback (\ref{feedback-red}) be chosen. Then the solutions \\$\varphi(t)=e^{(\mc A+ KBB^*)t}\varphi_0$  for $t\in \mathbb{R}^+$ of the closed-loop system (\ref{d4-non})-(\ref{divp-non})
     is exponentially stable in $\mc H.$
\end{theorem}

\begin{proof}To prove this, we replace Lemma 4.2  by Lemma \ref{xyz} in the  proof of Theorem 4.3 in \cite{Ozkan3}. The rest of the proof uses the compact perturbation argument the same way as in Theorem 4.3 in \cite{Ozkan3}.
\end{proof}

\subsection{The fully dynamic Mad-Marcus (M-M) model}
The model with magnetic effects is a strongly coupled system for bending, shear and charge equations. We  consider the  bending-free model with the following $B^*-$ type feedback controller $V(t)=-k_1  \dot p(L), \quad k_1>0:$
\begin{eqnarray}
 \label{perturbeddd-an-1} \left\{
  \begin{array}{ll}
  \mu  h_3 \ddot p   -\beta B_4  p_{xx}    -\gamma\beta h_2h_3 \varsigma B_2^2 J p = -V(t)\delta_L,&\\
  p(0)=0,~~ \beta B_4 p_x(L)=-k_1 \dot p(L),&\\
   ( p,  \dot p)(x,0)=( p_0,  p_1).
    \end{array} \right.
\end{eqnarray}
Let $\mc{H} =  H^1_L(0,L)\times \Ltwo(0,L).$ The  energy associated with (\ref{perturbeddd-an-1}) is
\begin{eqnarray}
\nonumber && \mathrm{E}(t) =\frac{1}{2}\int_0^L \left\{ \mu h_3  |\dot p|^2 +\beta B_4 |p_x|^2-\gamma \beta h_2 h_3 \varsigma B_2^2(J  p)) \bar p \right\} ~dx
\end{eqnarray}
 This motivates definition of the inner product on $\mc{H}$
\begin{eqnarray}
\nonumber  \left<\left[ \begin{array}{l}
 u_1 \\
 u_2
 \end{array} \right], \left[ \begin{array}{l}
 v_1 \\
 v_2
 \end{array} \right]\right>_{\mc{H}}=  \int_0^L \left\{\mu h_3  u_2 { {\bar v}_2}+ \beta B_4 (u_1)_x (\bar v_1)_x  -\gamma\beta \varsigma h_2 h_3 \tilde B^2 J (u_1) (\bar u_1) \right\}dx.
 \end{eqnarray}
Define  the operator $\mc A: {\text{Dom}}(\mc A)\subset \mc H \to \mc H:$
 \begin{eqnarray}
 \label{op-a}\mc A= \left[ {\begin{array}{*{20}c}
   0 & I \\
       \frac{\beta B_4}{\mu h_3} D_x^2    +\frac{\gamma\beta h_2\varsigma B_2^2}{\mu } J   &  0  \\
\end{array}} \right]
\end{eqnarray} where
\begin{eqnarray}
 \label{d-mm}  \left.
\begin{array}{ll}
{\rm {Dom}}{\mc A} = (H^2(0,L)\cap H^1_L(0,L))    \bigcap \{ \vec z :~ \beta B_4 (z_1)_{xx}    +\gamma\beta h_2h_3\varsigma B_2^2 J z_1  \in \mc H, & \\
   \beta B_4 (z_1)_x(L)    +k_1 \mu h_3 z_2(L)=0\}
 \end{array} \right.
\end{eqnarray}
 \begin{theorem} \label{skeww-1}The operator $\mc A$ defined by (\ref{op-a})-(\ref{d-mm}) is dissipative in $\mc H.$ Moreover,
$\mc A^{-1}$ exists and is compact on $\mc H.$ Therefore, $\mc A$ generates a $C_0$-semigroup of contractions on $\mc H$ and the
spectrum $\sigma(\mc A)$ consists of isolated eigenvalues only.
\end{theorem}
\begin{proof} Let $Y\in{\rm Dom}(\mc A).$ Then
\begin{eqnarray*}
\nonumber &&\left<\mc AY,Y\right>=\int_0^L \left[\left( \beta B_4 (y_1)_{xx}    +\gamma\beta h_2h_3\varsigma B_2^2 J y_1 \right) \bar y_2 \right.\\
\nonumber && \left.+ \left(\beta B_4 (y_2)_x (\bar y_1)_x -\gamma\beta \varsigma h_2 h_3 \tilde B^2 J (y_2) (\bar y_1)\right)\right]~dx\\
\nonumber =&&\left.   \beta B_4 (y_1)_{x}  \bar y_2\right|_{x=0}^L\\
\nonumber &&+\int_0^L\left[  \beta B_4 (-(y_1)_{x}(\bar y_2)_x +  (\bar y_1)_x) (y_2)_x +\gamma\beta h_2h_3\varsigma B_2^2 \left(J y_1 \bar y_2- J (\bar y_1)(y_2)\right) \right]dx
\end{eqnarray*}
Therefore, ${\rm Re}\left<\mc AY,Y\right>=-k_1\mu h_3|z_2(L)|^2\le 0.$ Therefore, $\mc A$ generates a $C_0$-semigroup of contractions on $\mc H.$

Next, we show that $0\in \sigma (\mc A),$ i.e. $0$ is not an eigenvalue. Let $Z\in \mc H.$ We show that there exists $Y\in {\rm Dom}(\mc A)$ such that $\mc A Y=Z:$
\begin{eqnarray*}
&& y_2=z_1\in H^1_L(0,L),\\
&&\beta B_4  (y_1)_{xx}    +\gamma\beta h_2h_3 \varsigma B_2^2 J y_1 =z_2.
\end{eqnarray*}
Since $Jy_1=P_\xi^{-1} (y_1)_{xx},$ the second equation  can be re-written as
\begin{eqnarray*}(\beta B_4  I +\gamma\beta h_2h_3 \varsigma B_2^2 P_\xi^{-1})(y_1)_{xx}=z_2
\end{eqnarray*} where $(\beta B_4  I +\gamma\beta h_2h_3 \varsigma B_2^2 P_\xi^{-1})$ is a positive operator, and therefore, is invertible. Therefore,
\begin{eqnarray*}(y_1)_{xx}=(\beta B_4  I +\gamma\beta h_2h_3 \varsigma B_2^2 P_\xi^{-1})^{-1}  z_2:=\psi\in L^2(0,1).
\end{eqnarray*}
By integrating the equation twice we conclude that
\begin{eqnarray*}
y_1=\int_1^x (x-\tau )\psi(\tau) d\tau + \int_0^1 \tau \psi(\tau) d\tau - \frac{k_1 \mu h_3z_1(L)}{\beta B_4}x.
\end{eqnarray*}
Thus, $Y\in {\rm Dom}(\mc A).$ Since  $0\in \sigma (\mc A),$ and $\mc A^{−-1}$ is compact on $\mc H,$   the
spectrum $\sigma(\mc A)$ consists of isolated eigenvalues only.
 \end{proof}

Let $\lambda =i \tau.$ Consider the eigenvalue problem $\mc A \varphi=\lambda \varphi$ corresponding to (\ref{perturbeddd-an-1}):
\begin{eqnarray}
\label{perturbed-dumb-2} \left\{ \begin{array}{ll}
  C\varsigma \phi^2 -\phi^2_{xx} + B_2 p_{xx}=0&\\
   - \beta B_4  p_{xx} + \gamma \beta h_2h_3 \varsigma B_2\phi^2   =  \mu  h_3 \tau^2 p,&\\
   \phi_2(0)=\phi^2_x(L)=p(0)=  \beta B_4 p_x(L)    +k_1 \mu h_3 i\tau  p(L)=0.
  \end{array} \right.
\end{eqnarray}
Now let $\Phi=[p~ p_x ~\varphi ~\varphi_x]^{\rm T}.$ The system (\ref{perturbed-dumb-2}) can be written as
\begin{eqnarray*}
\Phi_x=\tilde A\Phi:=\left(
                      \begin{array}{cccc}
                        0 & 1 & 0 & 0 \\
                        -\frac{\mu h_3 \tau^2}{\beta B_4} & 0 & \frac{\gamma h_2 h_3 \varsigma B_2}{B_4} & 0 \\
                        0 & 0 & 0 & 1 \\
                        -\frac{\mu h_3 B_2\tau^2}{\beta B_4}  & 0 & C\varsigma + \frac{\gamma h_2 h_3 \varsigma B_2^2}{B_4} & 0 \\
                      \end{array}
                    \right)\Phi.
\end{eqnarray*}
Now consider the auxiliary eigenvalue problem $\tilde {\mc A}  \tilde \Phi=\tilde \lambda \tilde \Phi $ that has the following characteristic equation:
\begin{eqnarray*}
\tilde \lambda^4+\left(\frac{h_3 \mu  \tau ^2}{\beta  B_4}-\frac{B_2^2 \gamma \varsigma h_2 h_3}{B_4}-C \varsigma\right)\tilde \lambda^2-\frac{C\varsigma  h_3 \mu  \tau ^2}{\beta  B_4}=0.
\end{eqnarray*}
Let  $\tau ^2>\frac{\beta B_4}{\mu h_3}\left(C\varsigma + \frac{B_2^2 \gamma \varsigma h_2 h_3}{B_4}\right)$. There are four complex conjugate eigenvalues $\tilde \lambda=\{\mp ia_1, \mp ia_2\}$ where
      \begin{eqnarray*}
      \nonumber a_1=\frac{\sqrt{\frac{-B_2^2 \gamma \varsigma h_2 h_3}{B_4}-C\varsigma+\frac{h_3 \mu  \tau ^2}{\beta  B_4}+\sqrt{(\frac{h_3 \mu  \tau ^2}{\beta  B_4}-\frac{B_2^2 \gamma \varsigma  h_2 h_3}{B_4}-C\varsigma)^2+4\frac{C \varsigma  h_3 \mu  \tau ^2}{\beta  B_4}}}}{\sqrt{2}}\\
    \label{e-values}  a_2=\frac{\sqrt{\frac{-B_2^2 \gamma \varsigma h_2 h_3}{B_4}-C\varsigma+\frac{h_3 \mu  \tau ^2}{\beta  B_4}- \sqrt{(\frac{h_3 \mu  \tau ^2}{\beta  B_4}-\frac{B_2^2 \gamma \varsigma h_2 h_3}{B_4}-C\varsigma)^2+4\frac{C \varsigma h_3 \mu  \tau ^2}{\beta  B_4}}}}{\sqrt{2}}.
      \end{eqnarray*}
      Define $b_1:=\frac{B_4}{\gamma \varsigma h_2 h_3 B_2}\left(\frac{\mu h_3}{\beta B_4}\tau^2-a_1^2\right),$ $b_2:=\frac{B_4}{\gamma\varsigma h_2 h_3 B_2}\left(\frac{\mu h_3}{\beta B_4}\tau^2-a_2^2\right)$ where $b_1,b_2\ne 0$ and $b_1-b_2\ne 0.$ By using the boundary conditions at $x=0,$
      \begin{eqnarray*}
     \nonumber && y_1(x)=\frac{a_1 \left(b_1 C_1-C_2\right) \sin \left(a_2 x\right)+a_2 \left(C_2-b_2 C_1\right) \sin \left(a_1 x\right)}{a_1 a_2 \left(b_1-b_2\right)}\\
     \nonumber &&y_2(x)=\frac{a_1 b_2 \left(b_1 C_1-C_2\right) \sin \left(a_2 x\right)+a_2 b_1 \left(C_2-b_2 C_1\right) \sin \left(a_1 x\right)}{a_1 a_2 \left(b_1-b_2\right)}
      \end{eqnarray*}
      where $C_1, C_2$ are arbitrary constants.  By using the boundary conditions at $x=L,$ the coefficient matrix for $(C_1 ~ C_2)^{\rm T}$ has the determinant
      {\small
      \begin{eqnarray*}
      \beta  B_4 \cos \left(a_1 L\right) \cos \left(a_2 L\right)+\frac{i h_3 K_1 \mu  \tau  \left(a_1 b_1 \sin \left(a_2 L\right) \cos \left(a_1 L\right)-a_2 b_2 \sin \left(a_1 L\right) \cos \left(a_2 L\right)\right)}{a_1 a_2 \left(b_1-b_2\right)}=0.
      \end{eqnarray*}}
      Assume $a_1=\frac{(2n-1)\pi}{2L}, a_2=\frac{(2m-1)\pi}{2L}$ for some $n,m\in \mathbb{R}^+$ so that
       \begin{eqnarray}\label{tau}\tau_{nm}=\sqrt{\frac{\beta B_4}{\mu h_3}\left(C\varsigma  + \frac{B_2^2 \gamma \varsigma h_2 h_3}{B_4}\right)+ (2n-1)^2+(2m-1)^2}.
        \end{eqnarray}
        Then, the determinant becomes zero. Therefore, we find a non-trivial solution of (\ref{perturbed-dumb-2}):
       \begin{eqnarray}\label{sol-eig}\left\{
      \begin{array}{cc} p_{nm}(x)=\sin \left(a_{2,m} x\right)-\sin \left(a_{1,n} x\right),&\\
         \phi_{nm}^2(x)=b_{2,m} \sin \left(a_{2,m} x\right)-b_{1,n} \sin \left(a_{1,n} x\right).
         \end{array}\right.
       \end{eqnarray}
       Note that the condition $a_1=\frac{(2n-1)\pi}{2L}, a_2=\frac{(2m-1)\pi}{2L}$  is equivalent to the following condition that material parameters satisfy
       \begin{eqnarray}
       \nonumber (2n-1)^2(2m-1)^2-C\varsigma \left[(2n-1)^2+(2m-1)^2\right]=\frac{C\varsigma L^2}{\pi^2}\left(C\varsigma  + \frac{B_2^2 \gamma \varsigma h_2 h_3}{B_4}\right).
       \end{eqnarray}
 \begin{theorem}\label{lack}
   Then the solutions $\{\varphi(t)\}_{t\in \mathbb{R}^+}$ with $V(t)=-k_1\dot p(L)$  of the closed-loop system (\ref{perturbeddd-an-1})
     is NOT strongly stable in $\mc H$ if $a_1=\frac{(2n-1)\pi}{2L}, a_2=\frac{(2m-1)\pi}{2L}$ for some $m,n\in\mathbb{Z},$ and $\tau_{nm}$ is defines by (\ref{tau})
\end{theorem}
\begin{proof}
Consider the eigenvalue problem (\ref{perturbed-dumb-2}) with (\ref{set2}). We show that
there are  eigenvalues on the imaginary axis, or in other words, the set
\begin{eqnarray} \label{set2}\left\{z\in \mc H: {\rm Re} \left<\mc A z, z\right>_{\mc H}= -k_1\mu h_3|z_2(L)|^2=0\right\}
\end{eqnarray}
has non-trivial solutions. With (\ref{set2}), the eigenvalue problem becomes overdetermined with the extra boundary condition   $p(L)=0.$ If $\frac{a_1}{a_2}=\frac{2n-1}{2m-1},$ the nontrivial solution (\ref{sol-eig}) automatically satisfies $p_{nm}(L)=0.$ In other words, the $B^*-$type feedback $V(t)=-k_1 \dot p_{nm}(L)$ does not stabilize the system if the material parameters satisfy $\frac{a_1}{a_2}=\frac{2n-1}{2m-1}$ for some $m,n\in\mathbb{Z}$ with $\tau_{nm}$ defined by (\ref{tau}).
\end{proof}
\subsection{Electrostatic Mead-Marcus (M-M) model}

The electrostatic M-M model (\ref{MM})-(\ref{MM-b}) is a continuous perturbation of the classical Euler-Bernoulli model due to the operator $J$  defined in Lemma \ref{pxi}. Controlling the Euler-Bernoulli beam through its boundary  has been a long standing problem in the PDE control theory, see \cite{Chen,Horn,Lasiecka,Liu} and the references therein. It is proved that one of the two controllers acting on the boundary is unnecessary to achieve exponential stability.

The only boundary feedback stabilization result for the model (\ref{MM})-(\ref{MM-b})   is provided by \cite{Wang} for a  three-layer composite (having no piezoelectric layer) with clamped-free boundary conditions, and only a mechanical controller $g(t)$ is applied at the free end $x=L,$ see (\ref{ivpd}).  This type of mechanical boundary control is ruled out for a smart piezoelectric M-M beam since we want to control the overall bending motion of the composite by only an electrical controller $V(t)$ which controls the bending moment at the tip $x=L$, not the transverse shear.



We choose the following $B^*-$ type feedback controller
\begin{eqnarray}\label{MM-feed}V(t)=-k_1  (P_\xi \dot w_x)(L,t),\quad  g(t)\equiv 0,~~~ \quad k_1>0
\end{eqnarray}
where $P_\xi \dot w_{x}(L)=\frac{1}{\tilde B \xi} \left( \dot w_x-\dot \phi^2 \right)(L),$ and $\dot \phi^2$ is the velocity of the shear of middle layer.   Presumably, this type of feedback is a perturbation of the angular velocity feedback $\dot w_x(L)$ in the Euler-Bernoulli model.
The energy of the system dissipates and it satisfies
\begin{eqnarray}
\nonumber\begin{array}{ll}
 \frac{dE(t)}{dt}&=\gamma V(t)\int_0^L \left[ h_2 h_3 \varsigma \tilde B (P_\xi \dot w_x)(L) +  \frac{B_3}{B_4} \dot w_x(L)\right]\\
 &=-\frac{k_1\gamma}{B_4} \left( h_2 h_3 \varsigma \tilde BB_2 P_\xi  +  B_3 I\right)\dot w_x(L) \cdot  P_\xi \dot w_x(L)\\
 &\le  0
\end{array}
\end{eqnarray}
where $ h_2 h_3 \varsigma \tilde BB_2 P_\xi  +  B_3 I$ is a non-negative operator. Observe that $P_\xi \dot w_x(L)\cdot \dot w_x(L)$  is the total piezoelectric effect  due to the coupling of the charge equation to shear and bending at the same time. In fact, this is a damping injection through the shear of the middle layer to control the bending moments at $x=L$.

Let the operator $\mc A$ be the same as (\ref{A-MM}) with the new domain
\begin{eqnarray}
 \label{A-MM-newd}  \left.
\begin{array}{ll}
{\rm {Dom}}(\mc A) = \{ (z_1,z_2)\in \mc H, z_2\in H^2_L(0,L),&\\
\quad ~\tilde A (z_1)_{xxx}+\gamma\beta \varsigma h_2 h_3 \tilde B^2  J (z_1)_x \in H^1(0,L),&\\
 \quad \left.\tilde A(z_1)_{xx}(L)+k_1\frac{\gamma}{B_4}\left( \varsigma h_2 h_3  \tilde B B_2 (P_\xi) + B_3 I\right)(z_2)_x\cdot P_\xi (\bar z_2)_x\right|_{x=L}=0,&\\
  \quad \tilde A (z_1)_{xxx}(L)+\gamma\beta \varsigma h_2 h_3 \tilde B^2  J (z_1)(L)=0 \}. &
 \end{array} \right.
\end{eqnarray}

\begin{theorem} \label{MM-mult}The operator $\mc A$ defined by (\ref{A-MM})-(\ref{A-MM-newd}) is dissipative in $\mc H.$ Moreover,
$\mc A^{-1}$ exists and is compact on $\mc H.$ Therefore, $\mc A$ generates a $C_0$-semigroup of contractions on $\mc H$ and the
spectrum $\sigma(\mc A)$ consists of isolated eigenvalues only.
\end{theorem}
\begin{proof} Let $Y\in{\rm Dom}(\mc A).$ Then
\begin{eqnarray*}
\nonumber \left<\mc AY,Y\right>&=&\int_0^L \left[\left(-\tilde A (y_1)_{xxxx} -\gamma\beta \varsigma h_2 h_3 \tilde B^2 D_x J (y_1)_x\right) \bar y_2  \right.\\
\nonumber && \left.\left(\tilde A (y_2)_{xx} (\bar y_1)_{xx}  -\gamma\beta \varsigma h_2 h_3 \tilde B^2 J (y_2)_{x}  (\bar y_1)_x \right)\right]~dx\\
\nonumber &=&\left. +\left(-\tilde A (y_1)_{xxx} -\gamma\beta \varsigma h_2 h_3 \tilde B^2 D_x J (y_1)\right) \bar y_2\right|_{x=0}^L+ \left.\tilde A (y_1)_{xx} (\bar y_2)_x\right|_{x=0}^L\\
\nonumber &&+\int_0^L\left[ -\tilde A (y_1)_{xx} (\bar y_2)_{xx}+ \gamma\beta \varsigma h_2 h_3 \tilde B^2 J (y_1)_x  (\bar y_2)_x \right.\\
\nonumber &&\left.\left(\tilde A (y_2)_{xx} (\bar y_1)_{xx}  -\gamma\beta \varsigma h_2 h_3 \tilde B^2 J (y_2)_{x}  (\bar y_1)_x \right)\right]~dx.
\end{eqnarray*}
Therefore, ${\rm Re}\left<\mc AY,Y\right>=-\frac{k_1\gamma}{B_4} \left( h_2 h_3 \varsigma \tilde BB_2 P_\xi  +  B_3 I\right) (z_2)_x(L) \cdot  P_\xi (\bar z_2)_x(L)\le 0.$
Therefore $\mc A$ is dissipative. Therefore,  if ${\mc A}^{-−1}$ exists, $\mc A$ must be densely defined in $\mc H.$ Therefore, $\mc A$ generates a $C_0$-semigroup of contractions on $\mc H.$ Next, we show that $0\in \sigma (\mc A),$ i.e. $0$ is not an eigenvalue. We solve the following problem:
\begin{eqnarray}
 \label{eig1-zero}\left\{\begin{array}{ll}
 {\tilde A} w_{xxxx}+\gamma\beta \varsigma h_2 h_3 \tilde B^2 (J w_{x})_x  = 0, &\\
 w(0)=w_x(0)= w_{x}(L)= {\tilde A} w_{xxx}(L)+\gamma\beta \varsigma h_2 h_3 \tilde B^2 J w_{x}(L)=0.
 \end{array}\right.
\end{eqnarray}
 Let $J w_x:=u.$ By the definition of $J=(\varsigma \tilde C I -D_x^2)^{-1} D_x^2,$ (\ref{eig1-zero}) is re-written as
  \begin{eqnarray}
\nonumber \left\{ \begin{array}{ll}
 {\tilde A} w_{xxxx} -  \beta \gamma h_2h_3 \varsigma{\tilde B}  u_x=0, &\\
\varsigma{\tilde C} u -u_{xx} + {\tilde B} w_{xxx} =0,&\\
w(0)=w_x(0)=u(0)= w_{xx}(L)=u_x(L)=  \tilde A w_{xxx}(L) - \beta \gamma h_2h_3 \varsigma{\tilde B} u(L)=0.
\end{array}\right.
 \end{eqnarray}
By using the last boundary condition, we integrate the first equation and plug it in the $u-$equation to get
$\left(\xi+\frac{\beta \gamma h_2h_3 \varsigma{\tilde B}^2  }{\tilde A}\right)u-u_{xx}=0.$
Since $\xi +\frac{\beta \gamma h_2h_3 \varsigma{\tilde B}^2  }{\tilde A} >0,$ by the boundary conditions for $u,$ we obtain that $u\equiv 0.$ This implies that $w_{xxx}=0.$ By the boundary conditions $w\equiv 0.$
$ {\tilde A} w_{xxx} -  \beta \gamma h_2h_3 \varsigma{\tilde B}  u=0.$
 Thus, $0\in \sigma (\mc A),$ and $\mc A^{-1}$ is compact on $\mc H.$  Hence the
spectrum $\sigma(\mc A)$ consists of isolated eigenvalues only.
\end{proof}
\begin{theorem}
   Let the feedback (\ref{MM-feed}) be chosen and $g(t)\equiv 0$ in (\ref{abstractMM}). Then the solutions $\varphi(t)$  for $t\in \mathbb{R}^+$ of the closed-loop system (\ref{Semigroupp-mmm})
     is strongly stable in $\mc H.$
\end{theorem}

\begin{proof}

If we can show that
there are no eigenvalues on the imaginary axis, or in other words, the set
\begin{eqnarray}\label{set} \left\{z\in \mc H: {\rm Re} \left<\mc A z, z\right>_{\mc H}= -\frac{k_1\gamma}{B_4} \left( h_2 h_3 \varsigma \tilde BB_2 P_\xi  +  B_3 I\right) (z_2)_x(L) \cdot (P_\xi\bar z_2)_x(L)=0\right\}
\end{eqnarray}
has only the trivial solution, i.e. $z=0$; then by La Salle's invariance principle,
the system is strongly stable. Therefore, proving the asymptotic stability of the
(1)-(3) reduces to showing that the following eigenvalue problem $\mc A z = \lambda z:$
\begin{eqnarray}
 \label{eig1}\left\{\begin{array}{ll}
 {\tilde A} w_{xxxx}+\gamma\beta \varsigma h_2 h_3 \tilde B^2 (J w_{x})_x +\lambda^2 w = 0, &\\
 w(0)=w_x(0)= w_{x}(L)= w_{xx}(L)=0, &\\
  {\tilde A} w_{xxx}(L)+\gamma\beta \varsigma h_2 h_3 \tilde B^2 J w_{x}(L)=(P_\xi w_x)(L)=0.
 \end{array}\right.
\end{eqnarray}
has only the trivial solution.
By using the definition of (\ref{MM}), i.e. $(J w_{x}) =(\varsigma P_\xi w_x)-  w_x,$ we obtain that $(J w_{x})(L)=0$  since both terms $(P_\xi w_x)(L) $ and $w_x(L)$ are zero by (\ref{set}).

Let $\lambda=i\omega$ where $\omega\in\mathbb{R}.$ Then  (\ref{eig1}) reduces to
\begin{eqnarray}
 \left\{\begin{array}{ll}
 {\tilde A} w_{xxxx}+\gamma\beta \varsigma h_2 h_3 \tilde B^2 (J w_{x})_x -\omega^2 w = 0, &\\
w(0)=w_x(0)= w_{x}(L)= w_{xx}(L)=w_{xxx}(L)=J w_{x}(L)=(P_\xi w_x)(L)=0.
 \end{array}\right.
\end{eqnarray}
Note that the following integrals are true.
\begin{eqnarray}
\nonumber && \int_0^L x w_{xxxx} \bar w_{xxx} dx=\frac{-1}{2}\int_0^L |w_{xxx}|^2dx, \\
\nonumber && \int_0^Lx  w \bar w_{xxx} dx= \int_0^L \frac{3}{2}\int_0^L |w_x|^2dx
\end{eqnarray}
and
\begin{eqnarray}
\nonumber && \int_0^L x(J w_x)_x\bar w_{xxx} dx= \int_0^L x((\xi P_\xi -I)w_x)_x  \bar w_{xxx}dx\\
\nonumber && \int_0^L \xi (P_\xi w_x)_x x\bar w_{xxx}dx +\frac{1}{2} \int_0^L |w_{xx}|^2dx
\end{eqnarray}
Let $z=P_\xi w_x.$ Then $\xi z -z_{xx}=w_x,$ and therefore
\begin{eqnarray}
\nonumber && \int_0^L \xi (P_\xi w_x)_x x\bar w_{xxx} dx =\int_0^L \xi z_x x (\xi \bar z_{xx}-\bar z_{xxxx})=\frac{-1}{2}\int_0^L \left(\xi^2|z_x|^2 +\xi |z_{xx}|^2\right) dx
\end{eqnarray}
Multiplying  the equation by $x\bar w_{xxx}$ and integrate by parts using the boundary conditions to obtain
\begin{eqnarray}
\nonumber \int_0^L \left(\tilde A |w_{xxx}|^2 +3 m |w_x|^2+ \xi^2|z_x|^2 +\xi |z_{xx}|^2\right) dx=0.
\end{eqnarray}
By using the overdetermined boundary conditions we obtain $w\equiv 0. \square$
\end{proof}

We state the following stability theorem and skip the proof since it goes beyond the scope of the paper. The proof uses the same type of frequency domain approach and spectral multipliers used in \cite[Theorem 4]{Ozkan6} where a stronger $B^*-$ type feedback is chosen $V(t)=- k_1\left[ \varsigma h_2 h_3  \tilde B B_2(P_\varsigma \dot w_x(L))  + B_3\dot w_x(L)\right]$ in comparison to (\ref{MM-feed}).
\begin{theorem}\label{MM-mult}
   Let the feedback (\ref{MM-feed}) be chosen and $g(t)\equiv 0$ in (\ref{abstractMM}). Then the solutions $\varphi$  for $t\in \mathbb{R}^+$ of the closed-loop system (\ref{Semigroupp-mmm})
     is exponentially stable in $\mc H.$
\end{theorem}
Note that our result not only confirms  the results in \cite{Baz} but also  improves them since only the asymptotic stability is mentioned  in \cite{Baz} without a proof.


   \section{Conclusion and Final Remarks}
 In this paper,  electrostatic voltage-controlled piezoelectric smart composite beam models are shown to be exponentially stable with the choice of the $B^*-$type state feedback, which are all mechanical. This is similar to the charge-actuation case but not the current actuation case where only the asymptotic stability can be achieved \cite{Ozer18b}. For the fully dynamic R-N model, asymptotic stability can be achieved for ``inertial sliding solutions" yet this is not the case for the  fully dynamic M-M model. There may still be eigenvalues on the imaginary axis. This implies that one electric controller for the piezoelectric layer may not be  enough in general to asymptotically (or exponentially) stabilize larger classes of solutions involving bending motions. This lines up with the results for the  charge or current-actuated models \cite{Ozer18b}. The stabilization  results are summarized in Table \ref{results}.

 Finally, we can conclude that even though the magnetic effects are minor in comparison to the mechanical and electrical effects for a piezoelectric layer, they have dramatic effects in controlling these composites. Note that
the stabilizability of fully dynamic R-N and M-M models for energy-space solutions is still an open problem. On the other hand, consideration of a remedial damping injection (by a mechanical feedback controller) to the piezoelectric layer of  the fully dynamic models is under consideration. Numerical results confirm that mechanical feedback controllers have a stronger effect to suppress vibrations \cite{Ozer18c}. Together with the effect of shear damping, the investigation of the optimal decay  rates to tune up the damping parameters and the feedback gains is the topic of future research.

The modeling of the three layer composition can be formed into a bimorph energy harvester \cite{Erturk,Shu} or a shear mode energy harvester  \cite{Sodano} to convert the stabilization problem to an energy harvesting problem.  The mathematical analysis provided in this paper will be a perfect foundation for future research on these models.

\begin{table}[htp]
\centering
\caption{Stability results for the closed-loop system with the $B^*-$feedback controller corresponding to the control $V(t)$ of the piezoelectric layer.}
\label{table}
\setlength{\tabcolsep}{3pt}
 \begin{tabular}{|p{60pt}|p{60pt}|p{155pt}|p{40pt}|}
    \hline
     Assumption  & Model& $B^*-$measurement for $V(t)$ at $x=L$& Stability  \\ \hline\hline
    E-static &  \multirow{ 2}{*}{Rao-Nakra} & Stretching \& compressing velocity & E.S.   \\
       F. Dynamic & & Induced current & A.S.   \\
    \hline
    \hline
    E-static & \multirow{ 2}{*}{Mead-Marcus} & Angular velocity (bending) + shear velocity (middle layer)& E. S.  \\
         F. Dynamic && Induced current &Not A.S.   \\
    \hline
          \multicolumn{4}{p{340pt}}{\footnotesize Different electro-magnetic assumptions for cantilevered R-N and M-M models. Here E.S.=Exponentially Stability for all modes, A.S.= Asymptotically Stability for inertial sliding solutions, Not A.S.=Not Asymptotically Stability for inertial sliding solutions.}
    \end{tabular}
\label{results}
\end{table}


\medskip
Received xxxx 20xx; revised xxxx 20xx.
\medskip

\end{document}